\documentclass[11pt]{article}
\usepackage{amssymb,amsmath,amsthm,graphicx}
\usepackage{mathrsfs}
\usepackage{color}
\usepackage{tikz}
\usepackage{subcaption}
\usepackage[colorlinks,linkcolor=blue,anchorcolor=blue,citecolor=blue]{hyperref}

\numberwithin{equation}{section}

\newtheorem{theorem}{Theorem}

\newtheorem{corollary}[theorem]{Corollary}

\newtheorem{lemma}[theorem]{Lemma}
\newtheorem{proposition}[theorem]{Proposition}
\newtheorem{remark}[theorem]{Remark}

\newcommand{\R}{\mathbb{R}}
\newcommand{\T}{\mathbb{T}}

\renewcommand{\P}{\mathbf{P}}

\newcommand{\ip}{(\mathbb{I}-\mathbf{P})}

\let\pt=\partial

\newcommand{\dd}{\mathrm{d}}
\let\d=\delta
\let\a=\alpha
\let\e=\varepsilon

\let\s=\sigma

\let\G=\Gamma

\let\g=\gamma
\let\th=\theta

\let\b=\beta

\newcommand\normmm[1]{\left\vert\kern-0.25ex \left\vert\kern-0.25ex \left\vert #1 \right\vert\kern-0.25ex \right\vert\kern-0.25ex \right\vert}
\newcommand\normm[1]{\left\lVert#1\right\rVert}
\newcommand\normmb[1]{\left[\kern-0.25ex\left[#1\right]\kern-0.25ex\right]}
\newcommand\norm[1]{\left\lvert#1\right\rvert}
\newcommand\inn[1]{\left\langle#1\right\rangle}

\newcommand\cQ{\mathcal{C}}

\begin{document}

\title{Regular Lenard-Balescu equations in a Periodic Box}
\author{Junhwa Jung\footnotemark[1] \and Toan T. Nguyen\footnotemark[1]}
\date{}

\maketitle

\footnotetext[1]{Penn State University, Department of Mathematics, State College, PA 16802, USA. Emails: jbj5730@psu.edu, nguyen@math.psu.edu. TN's research is supported in part by the NSF under grant DMS-2349981 and by a Simons fellowship.}

\begin{abstract}

The Lenard-Balescu equation is a collisional kinetic model widely used in plasma physics as a Bogoliubov correction to the meanfield Vlasov theory. Unlike the classical Landau and Boltzmann collision operators, the Lenard-Balescu collisional kernel not only accounts for the binary interaction between particles, but also includes the collective meanfield effects. In this paper, we construct global smooth solutions to the {\em regular} Lenard-Balescu equation near global Maxwellians in a periodic box, thus extending the previous work by Duerinckx-Winter that treats the spatially homogenous case to the inhomogenous setting. 

\end{abstract}

\tableofcontents 

\section{Introduction}

In plasma physics, the classical Landau collision operator is often used to describe collisions between particles. This describes however the precise binary interactions between particles whose travel as a free particle before collisions, while charged particles may be under influence by collective meanfield effects. This latter interaction can be described by the following Lenard-Balescu equation \cite{Balescu, Lenard}
\begin{align} \label{LenardBalescu}
\pt_t F + v \cdot \nabla_{x} F &=\cQ(F) \quad  \quad\text{ in } \quad\R_{+}\times \T^{d} \times \R^{d} \\ \label{LB equation initial}
F(t, x, v)|_{t=0}  &=   F_0(x,v) 
\quad \quad \text{ on }   \quad\T^d \times \R^d,
\end{align}
in which $F(t,x,v)$ represents the scalar density distribution of particles at time $t\geq 0$, position $x\in \T^d$ and velocity $v\in \mathbb{R}^d$, with $d\ge 2$. The Lenard-Balescu collision operator is defined by
\begin{equation} \label{LB - definition}
    \cQ(F):= \nabla_{v} \cdot \int_{\R^{d}} B(v,v-v_{*};\nabla_{v} F) (F_{*}\nabla_{v} F-F\nabla_{v_*} F_{*}) \dd v_{*},
\end{equation}
with $F = F(t,x,v)$ and $F_* = F(t,x,v_*)$, whose collision matrix kernel is computed by 
\begin{equation} \label{kernelB}
    B(v,w;\nabla_{v} F) := \int_{\R^{d}} (k \otimes k) |\widehat{V}(k)|^2 \frac{\delta(k \cdot w)}{\norm{\e (k,k\cdot v; \nabla_{v} F)}^2} \dd k,
\end{equation}
for any $v,w\in \R^d$. Here, $\delta(\cdot)$ is the one-dimensional Dirac delta function, $\widehat V(k)$ denotes the Fourier transform of pair-interaction potential functions $V(x)$, and $ \e (k,k\cdot v; \nabla_{v} F)$ is the Penrose dispersion function defined by
\begin{equation} \label{def-efunction}
    \e (k,k\cdot v; \nabla_{v} F) := 1+ \widehat{V}(k) \lim_{\gamma\to 0^+}\int_{\R^{d}} \frac{k \cdot \nabla_{v} F(v_{*})}{k\cdot (v-v_{*})-i\gamma} \dd v_{*}.
\end{equation}

The function $ \e (k,k\cdot v; \nabla_{v} F)$ is the classical dielectric function that plays a central role in the stability theory of plasmas \cite{Trivelpiece, Penrose}, accounting for the collective meanfield interactions between particles. The function is indeed responsible for Penrose stability and Landau damping of spatially homogenous equilibria; see, e.g., \cite{Landau, MV, Toan}. In the special case when $    \e (k,k\cdot v; \nabla_{v} F) \equiv 1$ (i.e. no meanfield effects), the collision kernel $
B(v,w;\nabla_{v} F)$ in \eqref{kernelB} reduces to $\frac{c_V}{|w|}(\mathbb{I} - \frac{w\otimes w}{|w|^2})$, namely the classical Landau collision operator (with Coulomb interaction as often referred to in the literature) \cite{AlexVillani, Villanireview} for some constant $c_V$ depending on the interaction potential $\widehat V(k)$. In general, the Penrose function $\e (k,k\cdot v; \nabla_{v} F)$ may vanish, and in fact always vanishes in the low frequency regime $|k|\ll 1$ for long-range interaction potentials \cite{Strain, Toan} (such as the Coulomb interaction). In this case, the collision kernel $B(v,w;\nabla_{v} F)$ may become more singular \cite{Strain}.

The Lenard-Balescu equation \eqref{LenardBalescu} retains the physical properties of the Boltzmann and Landau collisional models, including the conservation of mass, momentum, and energy. In addition, the celebrated $H$-theorem also holds for the Lenard-Balescu equation, and therefore one would expect the convergence to equilibrium remains valid. Inspired by the work of Guo \cite{guo2002landau} on the Landau equation near global Maxwellians and of Duerinckx-Winter \cite{duerinckx2023well} on the spatially homogenous case, we establish the global well-posedness theory for the spatially inhomogenous Lenard-Balescu equation \eqref{LenardBalescu}, focusing on the case when the collision kernel is {\em regular}. Specifically, we focus on the equation \eqref{LenardBalescu} in the case when the interaction potential $V$ is short range or screened, namely those potentials for which the Penrose function  $\e (k,k\cdot v; \nabla_{v} F) $ defined as in \eqref{def-efunction} never vanishes near global Maxwellians.

\subsection{Main Results }

Our main results read as follow. 

\begin{theorem}\label{theo-main}  (Global well-posedness close to equilibrium).
Let $d\ge 2$, and $V \in L^1(\R^d) \cap \dot{H}^{{2}}(\R^d)$ be symmetric and positive definite so that $x V \in L^2(\R^d)$. 
There exists a $v$-weighted, $L^2$-based Sobolev space involving derivatives up to order $d+7$, equipped with a norm $\normmm{\cdot}$, and a constant $\d_{0} >0$ such that if the initial data $F_0 = \mu + \sqrt{\mu} f_0 \geq 0$ satisfies
\begin{align*}
\normmm{f_{0}} \le \d_{0},
\end{align*}
then the nonlinear Lenard-Balescu equation \eqref{LenardBalescu} admits a unique global strong solution $F = \mu + \sqrt{\mu} f \geq 0$ satisfying the uniform bound
\begin{align}\label{0-uniform-bound}
\normmm{f(t)} \leq C \normmm{f_{0}}
\end{align}
for all $t\ge 0$ and for some universal constant $C>0$.
\end{theorem}

Given that the Lenard–Balescu equation satisfies an H-theorem, it is anticipated that the solutions will relax to a Maxwellian equilibrium state. Specifically, we obtain the following theorem providing a quantitative rate of convergence of solutions to the equilibrium in the large time. 

\begin{theorem} \label{theo-2} (Convergence to equilibrium).
Under the assumption of Theorem \ref{theo-main}, let $F = \mu + \sqrt{\mu} f$ be the constructed unique global solution of the Lenard-Balescu equation \eqref{LenardBalescu}. Then, the followings hold. 

\begin{enumerate}
    \item If $\iint_{\T^d \times \R^d} \inn{v}^{l} |f_{0}|^2\; \dd x \dd v < \infty$ for some $l > 0$ and the initial data is sufficiently small in the norm given in Theorem \ref{theo-main}, then we have for $\d < 1$:
    \begin{align*}
    \iint_{\T^d \times \R^d} |f(t)|^2 \; \dd x \dd v\lesssim \inn{t}^{-\d l}  \iint_{\T^d \times \R^d} \inn{v}^{l} |f_{0}|^2 \; \dd x \dd v.
    \end{align*}

    \item If $\iint_{\T^d \times \R^d} e^{K \inn{v}^{\th}} |f_{0}|^2\; \dd x \dd v < \infty$ for some $0 < \th <2$ and $K >0$, or $\th = 2$ and some small enough $K>0$ (only depending on $V$)  and the initial data is sufficiently small in the norm given in Theorem \ref{theo-main}, then we have,
    \begin{align*}
    \iint_{\T^d \times \R^d} |f(t)|^2\; \dd x \dd v \lesssim \exp \left( - \frac{K}{C} t^{\frac{\th}{\th+1}}\right)  \iint_{\T^d \times \R^d} e^{K \inn{v}^{\th}} |f_{0}|^2\; \dd x \dd v,
    \end{align*}
    for some constant $C$.
\end{enumerate}

\end{theorem}

Theorems \ref{theo-main}-\ref{theo-2} provide the boundedness and convergence of solutions in a suitable energy and dissipation norm $\normmm{f(t)}$, see \eqref{normmm -definition} below, and thereby establishes the global-in-time existence and the relaxation to the equilibrium of solutions to the spatially inhomogeneous Lenard-Balescu equation \eqref{LenardBalescu}, thus extending the previous work by Duerinckx-Winter that treats the homogeneous case in \cite{duerinckx2023well}. The extension to the inhomogeneous setting faces several fundamental difficulties. Most notably, the transport term $v \cdot \nabla_x$ and the functional dependence of the dispersion function $\e(k,k\cdot v; \nabla_v F)$ on the distribution $F$ require a carefully designed weighted norm. This structure is essential to achieve the consistent control over the linear and nonlinear operators, particularly in managing the derivative loss inherent to the Lenard-Balescu kernel. Let us mention a few points.

\subsubsection*{Control of the Inhomogeneous Collision Operator}

Unlike the classical Landau equation, the Lenard-Balescu operator involves a dynamic dispersion function $\e$. In the inhomogeneous setting, we must control not only the velocity derivatives but also the spatial derivatives of $\e$. Building upon the framework in \cite{duerinckx2023well}, Sections \ref{sec-Pre} and \ref{sec-nonop} provide an extensive analysis of the differences between the homogeneous and inhomogeneous kernels. We particularly focus on the weighted estimates required for both linear and nonlinear terms, which were not present in previous homogeneous studies.

\subsubsection*{Transport Term and Macro-Micro Decomposition}

A second major difficulty arises from the transport term $v \cdot \nabla_x F$ and its interaction with the kernel of the linearized operator $\mathcal{L}$. Following \cite{duerinckx2023well}, the linearized Lenard-Balescu and Landau operators share identical properties in the $V \in L^2, \dot{H}^2$ regime. Consequently, standard energy estimates yield dissipation only for the microscopic part (outside the kernel). In the homogeneous case, the kernel can be neutralized via Galilean scaling; however, this is impossible globally in an inhomogeneous space. To resolve this, we employ the test function method developed in \cite{esposito2013non, Esposito2017, jung2025global} in designing specific elliptic test functions corresponding to the macroscopic fluid components to recover the estimates for the kernel parts. This process is detailed in Section \ref{sec-L2}.

\subsubsection*{Weighted Norm Structure and Derivative Loss}

The most technical aspect of our proof lies in the design of the weighted norm structure. In the study of the Landau equation in a periodic box, Guo introduced weights depending on the number of velocity derivatives to control the transport term within the dissipation norm \cite{guo2002landau}. Our problem requires a similar but more complex structure. In the Lenard-Balescu equation, the dependence of the collision kernel on $F$ through the dispersion function causes a "loss" of one velocity derivative when estimating nonlinear terms—a phenomenon absent in the Landau case. To close the bootstrap argument, we implement a weight $\inn{v}^{-\norm{\a}-2\norm{\b}+N}$ that depends on both spatial ($\alpha$) and velocity ($\beta$) derivatives. More specifically, we shall propagate the weighted norms 
\begin{equation}\label{fnorm}\normm{\inn{v}^{N - \norm{\a}-2\norm{\b}}\partial_x^\alpha \partial_v^\beta f(t)}_{L^2_{x,v}} \end{equation}
for all $|\alpha|+|\beta|\le N$ with $N = d+7$ that control the energy of $f$. A similar norm is introduced to control the dissipation of $f$. See Section \ref{sec-norm} for the details. 

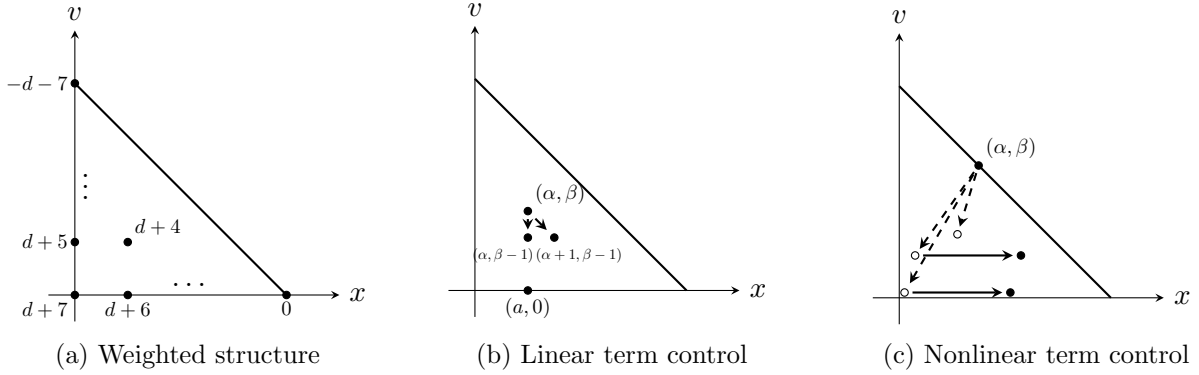
\begin{figure}[htbp]
    \centering
    
    % --- (a) Weighted Structure ---
    \begin{subfigure}[b]{0.32\textwidth}
        \centering
        \begin{tikzpicture}[scale=0.7] 
            \draw[->, >=stealth] (-0.5, 0) -- (5, 0) node[right] {$x$};
            \draw[->, >=stealth] (0, -0.5) -- (0, 5) node[above] {$v$};
            %\node[below left] at (0,0) {$O$};
            
            \draw[thick] (0, 4) -- (4, 0); 
            
            \foreach \p in {(0,0), (1,0), (0,1), (1,1), (4,0), (0,4)} {
                \filldraw [black] \p circle (2pt);
            }

            \node[below left, scale=0.7] at (0,0) {$d+7$};
            \node[below, scale=0.7] at (1,0) {$d+6$};
            \node[left, scale=0.7] at (0,1) {$d+5$};
            \node[above right, scale=0.7] at (1,1) {$d+4$};
            \node[below, scale=0.7] at (4,0) {$0$};
            \node[left, scale=0.7] at (0,4) {$-d-7$};

            \node at (2.2, 0.2) {$\dots$};
            \node at (0.2, 2.2) {$\vdots$};
%            \node at (2.0, 1.3) {$\ddots$};
        \end{tikzpicture}
        \caption{Weighted structure}
        \label{fig:graph_a}
    \end{subfigure}
    \hfill 
    % --- (b) Linear Control ---
    \begin{subfigure}[b]{0.32\textwidth}
        \centering
       \begin{tikzpicture}[scale=0.7]
            \draw[->, >=stealth] (-0.5, 0) -- (5, 0) node[right] {$x$};
            \draw[->, >=stealth] (0, -0.5) -- (0, 5) node[above] {$v$};
           % \node[below left] at (0,0) {$O$};
            
            \draw[thick] (0, 4) -- (4, 0);
            
          %  \foreach \p in {(0,0), (0.5,0), (0,0.5), (0.5,0.5), (4,0), (0,4)} {
          %      \filldraw [black] \p circle (2pt); }
            
%            \node at (2.2, 0.2) {$\dots$};
%            \node at (0.2, 2.2) {$\vdots$};
%            \node at (2.0, 1.3) {$\ddots$};

            \coordinate (A) at (1, 1.5);
            \coordinate (B) at (1, 1);
            \coordinate (C) at (1.5, 1);
            \coordinate (D) at (1, 0);
            \filldraw [black] (A) circle (2pt) node[above right, scale=0.7] {$(\a,\b)$};
            \filldraw [black] (B) circle (2pt) node[shift={(215:0.4)}, scale=0.5] {$(\a,\b-1)$};
            \filldraw [black] (C) circle (2pt) node[shift={(325:0.4)}, scale=0.5] {$(\a+1,\b-1)$};
            \filldraw [black] (D) circle (2pt) node[below, scale=0.7] {$(a,0)$};
            \draw[->, >=stealth, shorten >=3pt, shorten <=3pt, thick] (A) -- (B);
            \draw[->, >=stealth, shorten >=4pt, shorten <=4pt, thick] (A) -- (C);
        \end{tikzpicture}
        \caption{Linear term control}
        \label{fig:graph_b}
    \end{subfigure}
    \hfill 
    % --- (c) Nonlinear Control ---
     \begin{subfigure}[b]{0.32\textwidth}
        \centering
       \begin{tikzpicture}[scale=0.7]
            \draw[->, >=stealth] (-0.5, 0) -- (5, 0) node[right] {$x$};
            \draw[->, >=stealth] (0, -0.5) -- (0, 5) node[above] {$v$};
           % \node[below left] at (0,0) {$O$};
            
            \draw[thick] (0, 4) -- (4, 0);
            
          %  \foreach \p in {(0,0), (0.5,0), (0,0.5), (0.5,0.5), (4,0), (0,4)} {
          %      \filldraw [black] \p circle (2pt); }
            
%            \node at (2.2, 0.2) {$\dots$};
%            \node at (0.2, 2.2) {$\vdots$};
%            \node at (2.0, 1.3) {$\ddots$};

            \coordinate (A) at (1.5, 2.5);
            \coordinate (B) at (0.3, 0.8);
            \coordinate (C) at (0.1, 0.1);
            \coordinate (D) at (1.1, 1.2);
            \coordinate (E) at (2.1, 0.1);
            \coordinate (F) at (2.3, 0.8);
            \filldraw [black] (A) circle (2pt) node[above right, scale=0.7] {$(\a,\b)$};
            \draw [black] (B) circle (2pt);
            \draw [black] (C) circle (2pt);
            \draw [black] (D) circle (2pt);
            \filldraw [black] (E) circle (2pt);
            \filldraw [black] (F) circle (2pt);
            \draw[->, >=stealth, dashed, shorten >=3pt, shorten <=3pt, thick] (A) -- (B);
            \draw[->, >=stealth, dashed, shorten >=4pt, shorten <=4pt, thick] (A) -- (C);
            \draw[->, >=stealth, dashed, shorten >=3pt, shorten <=3pt, thick] (A) -- (D);
            \draw[->, >=stealth, shorten >=3pt, shorten <=3pt, thick] (C) -- (E);
            \draw[->, >=stealth, shorten >=3pt, shorten <=3pt, thick] (B) -- (F);
        \end{tikzpicture}
        \caption{Nonlinear term control}
        \label{fig:graph_c}
    \end{subfigure}

    \caption{Visualization of the norm structure and the associated term-wise controls.}
    \label{fig:main_framework}
\end{figure}

\subsubsection*{Graphical Illustration of the Control Mechanism}

The interplay between the weighted structure and the operator control is summarized in Figure \ref{fig:main_framework}. Precisely, 

\begin{itemize}

\item Figure \ref{fig:graph_a}: Weighted Norm Distribution. Each node represents the number of spatial and velocity derivatives, respectively. The values indicate the power of the weight. We observe that the weight increases with the order of derivatives, but vanishes for pure spatial derivatives, particularly at the highest order $(d+7, 0)$, to allow for the $L^2$ energy balance.

\item Figure \ref{fig:graph_b}: Linear Term Control. The transport term $v \cdot \nabla_x$ acts by increasing one spatial derivative while decreasing one velocity derivative. Due to our weight design, this shift gains additional weight, allowing the term to be bounded by the dissipation norm (left arrow). Simultaneously, the linearized operator is controlled by a norm with one fewer velocity derivative (right arrow). In the absence of velocity derivatives, all associated weights are non-positive. By applying Lemma \ref{lemma:Coerceivity of L}, we can reduce the weighted estimates to a weightless setting, which in turn allows us to gain control over the microscopic $\ip F$ parts. As for the macroscopic $\P F$ parts, their estimates are established via Proposition \ref{proposition:0-macro-L2-estimate}.

\item Figure \ref{fig:graph_c}: Nonlinear Term Control. Due to the presence of cubic and higher-order terms, we utilize Sobolev embeddings into $L^\infty$ over the torus $\mathbb{T}^d$. The high-order derivatives are distributed (dashed lines) such that lower-order terms are controlled via embedding (solid lines), effectively balancing the energy and dissipation norms to close the estimates.

\end{itemize}

%%%%%%%%%%%%%%%%%%%
%
%
%
%
%
%
%%%%%%%%%%%%%%%%%%%%555

\subsection{Notations}
For notational simplicity, we denote by $\langle \cdot, \cdot \rangle$ the standard $L^2$ inner product on $\mathbb{R}^d_v$, with its corresponding norm $\norm{\cdot}_{L^2_{v}}$. We further denote by $\normm{\cdot}_{L^p_{x}L^{q}_{v}}$ the $L^p$ norm over $x$ and $L^q$ norm over $v$ variables in the domain $\mathbb{T}^d \times \mathbb{R}^d$. Let $\a$, $\b$ denote multi-indices with length $\norm{\a}$, and $\norm{\b}$ respectively and we define
\begin{equation*}
\pt^{\a}_{\b} := \pt^{\a}_{x} \pt^{\b}_{v}, \qquad \pt^{\a} := \pt^{\a}_{x}, \qquad \pt_{\b} := \pt^{\b}_{v}.
\end{equation*}
If each component of $\b$ is not greater than that of $\a$'s, we denote it by $\b \le \a$. We also define $\b < \a$ if $\b \le \g$ and $\norm{\b} < \norm{\a}$. In addition, throughout the paper, we use the notation $A\lesssim B$ to denote the bound $A\le C B$ for some universal constant $C$. Similarly, $A\approx B$ to denote $A\lesssim B$ and $B\lesssim A$. 

\subsection{Organization of the Paper}

The remainder of this paper is organized as follows: Section \ref{sec-Pre} analyzes the properties of the dispersion function and operators in the weighted inhomogeneous setting, while Section \ref{sec-nonop} focuses on the nonlinear collision operator. Section \ref{sec-L2} implements the test function method for macroscopic control. Finally, Section \ref{pf of main thm} constructs the global weighted norm and completes the proof of global existence via a bootstrap argument, while Section \ref{pf of thm2} establishes the convergence to the equilibrium.

\section{Preliminaries}\label{sec-Pre}

\subsection{Perturbations}

Let $\mu = \pi^{-\frac{d}{2}}e^{-{|v|^2}}$ be the global Maxwellian. We shall construct global solutions to \eqref{LenardBalescu} of the perturbed form 
\begin{equation}\label{AnsatzF}
F = \mu + \sqrt{\mu} f. 
\end{equation}
In view of \eqref{kernelB}, we note that $B(v,v-v_{*};\nabla_{v} F)(v-v_{*}) = 0$, and therefore the Maxwellian $\mu$ is an exact solution to \eqref{LenardBalescu}. In addition, by using the above Ansatz for $F$, it follows that
\begin{equation} \label{LBeqs}
    \pt_t f + v \cdot \nabla_{x} f = L[f] + N(f)
\end{equation}
where
\begin{equation} \label{defLN}
\begin{aligned}
    L[f]:=& (\nabla_{v} - v) \cdot \int_{\R^{d}} B(v,v-v_{*};\nabla_{v} \mu) (\sqrt{\mu_{*}}(\nabla_{v} + v)f - \sqrt{\mu}((\nabla_{v} +v)f)_{*}) \sqrt{\mu_*}\;\dd v_{*}, \\
    N(f):=& (\nabla_{v} - v) \cdot \int_{\R^{d}} B(v,v-v_{*};\nabla_{v} F_f) (f_{*}\nabla_{v} f - f(\nabla_{v} f)_{*} )\sqrt{\mu_*}\;\dd v_{*} \\
    &+ (\nabla_{v} - v) \cdot \int_{\R^{d}} (B(v,v-v_{*};\nabla_{v} F_f)  - B(v,v-v_{*};\nabla_{v} \mu)) \\
    & \quad \quad \quad \quad \quad \quad \quad \times (\sqrt{\mu_{*}}(\nabla_{v} + v)f - \sqrt{\mu}((\nabla_{v} +v)f)_{*})\sqrt{\mu_*}\;\dd v_{*},
\end{aligned}
\end{equation}
collecting linear and nonlinear terms, respectively. In the above, $F_f = \mu + \sqrt{\mu} f$ as in \eqref{AnsatzF}. Our nonlinear analysis applies directly to the perturbed equation \eqref{LBeqs} with small initial perturbations $f_0(x,v)$. Observe that the linearized operator has a non-trivial kernel of the $d+2$-dimensional subspace of $L^2(\mathbb{R}^d)$ given by
\begin{equation}\label{ker-P}
\ker L = \operatorname{span} \left \{1,   v, |v|^{2} \right \}\sqrt{\mu}.
\end{equation}
In addition, an orthonormal basis for $\ker L$ is $\{\chi_{i}\}_{i=0}^{d+1}$, where
\begin{align}\label{base-hat-chi}
&\chi_{0} := \sqrt{{\mu}}, \quad
\chi_{i} := \sqrt{2}v_i\sqrt{{\mu}} \;\; (i=1,\cdots, d),\quad
\chi_{d+1} := \frac{2\norm{v}^2-d}{\sqrt{2d}}\sqrt{{\mu}}.
\end{align}
The orthogonal projection of $f$ onto $\ker L$ is denoted by
\begin{equation}\label{Pf-abc-hat}
  \P f  =  a[f] \chi_{0} + \sum_{i=1}^{d}b_{i}[f] \chi_{i} + c[f] \chi_{d+1},
\end{equation}
with coefficients
\begin{equation}\label{hat-abc-def}
a[f]:=\inn{\chi_{0}, f },\quad
b_i[f]:=\inn{\chi_{i}, f} \;\; (i=1,\cdots, d),\quad
c[f]:=\inn{\chi_{d+1}, f}.
\end{equation}
Let $(\mathbb{I}-\mathbf{P})f$ denote projection onto the orthogonal
complement of $\ker L $. 

\subsection{Operators}
For sake of convenience, for each function $F$, we introduce 
\begin{equation} \label{defcB}
\mathcal{B}_{\nabla_{v} F}[g] (v):= \int_{\R^{d}_{v_{*}}} B(v,v-v_{*};\nabla_{v} F) \sqrt{\mu_{*}} g_{*} \dd v_{*}.
\end{equation}
We may write the linear and nonlinear terms $L[g]$ and $N(g)$ in terms of $\mathcal{B}_{\nabla_{v} F}[\cdot]$. Indeed, we write the linearized operator $L$ as 
\begin{equation} \label{splitL}
L[g] = (\nabla_{v} -v) \cdot A (\nabla_{v} +v) g - (\nabla_{v} - v) \cdot (\sqrt{\mu} \mathcal{B}_{0}[(\nabla_{v} + v)g] ),
\end{equation}
where 
\begin{equation} \label{defAB}
\begin{aligned}
    A(v): &=  \mathcal{B}_{\nabla_{v} \mu}[\sqrt{\mu}] ,
   \qquad 
    \mathcal{B}_{0}[g](v) :=  \mathcal{B}_{\nabla_{v} \mu}[g] .
\end{aligned}
\end{equation}
Similarly, the nonlinear terms $N(g)$ can be written as 
\begin{equation}
\begin{aligned}
    N(g) = N(g,g,g)
\end{aligned}
\end{equation}
where
\begin{equation} \label{defNg123}
\begin{aligned}
    N(g_{1},g_{2},g_{3}) = & (\nabla_{v} - v) \cdot \Big( \mathcal{B}_{\nabla_{v} F_{g_{1}}}[g_{2}] \nabla_{v} g_{3} -\mathcal{B}_{\nabla_{v} F_{g_{1}}} [\nabla_{v} g_{2}] g_{3} \Big) \\
    & + (\nabla_{v} -v) \cdot \Big ( \mathcal{B}_{\nabla_{v} F_{g_{1}}}[\sqrt{\mu}] - \mathcal{B}_{\nabla_{v} \mu} [\sqrt{\mu}] \Big) (\nabla_{v} + v)g_{3} \\
    &- (\nabla_{v} -v) \cdot \Bigg(\sqrt{\mu} \Big( \mathcal{B}_{\nabla_{v} F_{g_{1}}} [(\nabla_{v} +v) g_{2}] - \mathcal{B}_{\nabla_{v} \mu} [(\nabla_{v} +v) g_{2}] \Big)\Bigg)
\end{aligned}
\end{equation}
with $F_g = \mu + \sqrt{\mu} g$.

\subsection{Energy Norms}\label{sec-norm}

In this section, we introduce the iterative energy and dissipation norms of $f$ that shall be used throughout the paper. Precisely, fix $N,M\ge 0$ and $r\in \R$, and introduce 
\begin{equation} \label{weighted norm definition}
\norm{f}^2_{N,M,r,L^2} := \sum_{\norm{\a} \le N, \norm{\b} \le M}\int_{\R^d} \norm{\inn{v}^{-r} \pt^{\a}_{\b} f}^{2} \dd v,\qquad \normm{f}^2_{N,M,r,L^2} := \int_{\T^d} \norm{f}^2_{N,M,r,L^2} \dd x, 
\end{equation}
for any scalar functions $f$, where $\pt^{\a}_{\b} = \pt^{\a}_{x} \pt^{\b}_{v}$ for any multi-indexes $\a,\b$. In view of \eqref{splitL}, we also introduce the $A$-norm of the vector field $\mathbf{f}$ as follows:
\begin{equation} \label{weighted A norm definition}
\begin{aligned}
\norm{\mathbf{f}}^2_{N,M,r,A} := \sum_{\norm{\a} \le N, \norm{\b} \le M}\int_{\R^d} \inn{v}^{-2r} \pt^{\a}_{\b} \mathbf{f} A \cdot \pt^{\a}_{\b} \mathbf{f} \dd v,\qquad \normm{\mathbf{f}}^2_{N,M,r,A} := \int_{\T^d} \norm{\mathbf{f}}^2_{N,M,r,A} \dd x,
\end{aligned}
\end{equation}
where the matrix $A$ is defined in \eqref{defAB}.
Next, we introduce the following dissipation norm
\begin{equation} \label{weighted D norm definition}
\begin{aligned}
\norm{f}^2_{N,M,r,D} := \norm{v f }^2_{N,M,r,A} + \norm{\nabla_{v} f}^2_{N,M,r,A}, \qquad \normm{f}^2_{N,M,r,D} := \int_{\T^d} \norm{f}^2_{N,M,r,D} \dd x.
\end{aligned}
\end{equation}
In particular, in case when $N,M,r$ are zeros, we simply write $\norm{\cdot}_{A} = \norm{\cdot}_{0,0,0,A}$ and $\norm{\cdot}_{D} = \norm{\cdot}_{0,0,0,D}$. Finally, we introduce the following iterative energy and dissipation functionals
\begin{align} \label{energy definition}
{e}_{N}[f](t) := & \sum_{N_{1}+N_{2} \le N} \normm{f(t)}_{N_{1},N_{2},-N+N_{1}+2N_{2},L^2}^{2},\\
\label{sup energy definition}
\mathscr{E}_{N}[f](t) := & \sup_{0\le s\le t} {e}_{N}[f](t),\\
\label{dissipation definition}
{d}_{N}[f](t) := & \sum_{N_{1}+N_{2} \le N}  \normm{f(t)}_{N_{1},N_{2},-N+N_{1}+2N_{2},D}^{2} ,\\
\label{integral dissipation definition}
\mathscr{D}_{N}[f](t) := & \int_{0}^{t}  {d}_{N}[f](s) \dd s.
\end{align}
The total energy functional is then defined by
\begin{equation} \label{normmm -definition}
\begin{aligned}
 \normmm{f}(t) :=&\sup_{s\in [0,t]}\mathscr{E}_{d+7}^{\frac{1}{2}}[f](s)  + \mathscr{D}_{d+7}^{\frac{1}{2}}[f](t).
\end{aligned}
\end{equation}
We emphasize that the energy and dissipation norms involve both decay and growth weights in $v$, since $N - N_1 - 2N_2$ may change the sign for $N_1 + N_2 \le N$. In particular, we note that the norm with top spatial derivatives $N_1 = N$ (and $N_2=0$) involves no $v$-weights, while the norm with top $v$-derivatives $N_2=N$ (and $N_1=0$) goes with weights $\langle v\rangle^{-N}$. The weights are distributed accordingly with derivatives $\langle v\rangle^{-|\a|}\pt^\a$ and $\langle v\rangle^{-2|\b|}\pt_\b$. On the other hand, as in the case with the Landau collision operator \cite{guo2002landau}, the dissipation norm naturally goes with the diffusion matrix $A$ which may become degenerate when $v$ is large, see \eqref{equivA}-\eqref{equivD} below.

\subsection{Linearized Collision Operator}

In this subsection, we study the linearized Lenard-Balescu collision operator. In fact, since the collision operator is autonomous in $x$ (i.e. only depending on $x$ through $f$), the results that were first established in \cite{duerinckx2023well} for the homogenous case carry over to the present inhomogenous case, which we shall recall below.  
 
\begin{lemma} \label{lem-disp}
(Lenard-Balescu dispersion function, \cite{duerinckx2023well}).
Let $V \in L^1 (\R^d)$.

(i) Non-degeneracy at Maxwellian: for all $k,v \in \R^{d}$
\begin{equation}
    \norm{\e(k,k\cdot v; \nabla_{v} \mu)} \approx_{V} 1.
\end{equation}

(ii) Non-degeneracy close to Maxwellian: Provided $g(x,\cdot)\in L^2(\R^{d})$ satisfies the following smallness condition, for some $r_0\ge0$, $\delta_0>0$, and some large enough constant $C_0$,

\begin{equation}
    \norm{\inn{v}^{-r_{0}} \inn{\nabla_{v}}^{\frac{3}{2}+\d_{0}}g(x,\cdot)}_{L^2_{v} } \le \frac{1}{C_0},
\end{equation}
we have for all $k,v\in\R^{d}$ and for $F_g = \mu + \sqrt{\mu} g$, 
\begin{equation}
\norm{\e(k,k\cdot v;\nabla_{v} F_g)}\,\approx_{V,\delta_0,r_0}\,1.
\end{equation}

(iii) Boundedness: For all multi-indices $\alpha>0$, $\beta >0$ for all $\delta>0$, and $r\ge0$, we have
\begin{equation} \label{epsilon - velocity derivative - estimate}
\norm{\pt_{\b}  \e(k,k\cdot v;\nabla_{v} F_g)}\,\lesssim_{V,\b, \delta,r}\, 1+\norm{\inn{v}^{-r}\inn{\nabla_{v}}^{|\b|+\frac{3}{2}+\delta}g(x,\cdot)}_{L^2_{v}},
\end{equation}
\begin{equation} \label
{epsilon - space and velocity derivative - estimate}
\norm{\pt_{\b}^{\a} \e(k,k\cdot v;\nabla_{v} F_g)}\,\lesssim_{V,\alpha,\b, \delta,r}\,\norm{\inn{v}^{-r}\inn{\nabla_{v}}^{|\b|+\frac{3}{2}+\delta}\pt^{\a} g(x,\cdot)}_{L^2_{v}},
\end{equation}
with $\pt_{\b}^{\a} = \pt_x^\a \pt_v^\b$.  

(iv) Boundedness of difference: For all multi-indices $\alpha, \beta \ge 0$, for all $\delta>0$, and $r\ge0$, we have
\begin{equation} \label{epsilon differemce - velocity derivative - estimate}
\norm{\pt_{\b}^{\a} (\e(k,k\cdot v;\nabla_{v} F_g) - \e(k,k\cdot v;\nabla_{v} \mu))}\,\lesssim_{V,\alpha,\b, \delta,r}\,\norm{\inn{v}^{-r}\inn{\nabla_{v}}^{|\b|+\frac{3}{2}+\delta}\pt^{\a} g(x,\cdot)}_{L^2_{v}},
\end{equation}
with $\pt_{\b}^{\a} = \pt_x^\a \pt_v^\b$.
\end{lemma}

\begin{proof} For the (i),(ii) and (iii), the lemma were established in lemma 2.1 of \cite{duerinckx2023well}, upon noting that the dispersion function $\e(k,k\cdot v;\nabla_{v} F_g)$ is defined pointwise in $x$. Finally, as for the (iv), recalling $F_g = \mu + \sqrt{\mu} g$, we note that 
\begin{align*}
\e(k,k\cdot v;\nabla_{v} F_g) - \e(k,k\cdot v;\nabla_{v} \mu) = \widehat{V}(k) \lim_{\gamma\to 0^+}\int_{\R^{d}} \frac{k \cdot \nabla_{v_{*}} (\sqrt{\mu}g)_{*}}{k\cdot (v-v_{*})-i\gamma} \dd v_{*}.
\end{align*}
The above expression depends only on $g$. Thus, property (iv) follows by applying the same argument used in Step 2 of the proof of Lemma 2.1 in \cite{duerinckx2023well}.
\end{proof}

\begin{lemma} \label{lemma:A - properties}
(Properties of $A$, \cite{duerinckx2023well}) 
Let $V\in L^1(\R^d)\cap\dot H^\frac12(\R^{d})$ be symmetric.
\begin{enumerate}
\item  {Coercivity and boundedness:} For all $v,e\in\R^{d}$,
\begin{equation*}
e\cdot A(v)e \,\approx_V\,\langle v\rangle^{-1} |P_v^\bot e|^2+\langle v\rangle^{-3}|P_ve|^2,
\end{equation*}
in terms of orthogonal projections $P_v$ and $P_v^\bot$ onto $v$ and $v^\bot$, namely 
\begin{equation}\label{eq:def-projv}
P_v:=\tfrac{v}{|v|}\otimes\tfrac{v}{|v|},\qquad P^\bot_v=\mathbb{I}-P_v.
\end{equation}
\item  {Smoothness:} the matrix 
$A$ belongs to $C^\infty_b(\R^{d})$ and satisfies for all $v\in\R^{d}$ and $\b \ge0$,
\begin{equation}\label{eq:smooth-A}
|\nabla_{v}^\b A(v)|\,\lesssim_{V,\b}\,\langle v\rangle^{-1},\qquad |\nabla_{v}^\b (A(v)  v)| \,\lesssim_{V,\b}\, \langle v\rangle^{-2}.
\end{equation}
In particular, for all vector fields $\mathbf{h}_1,\mathbf{h}_2$, and $\alpha\ge0$,
\begin{equation}\label{eq:hnabAh}
\norm{\int_{\R^{d}}\mathbf{h}_1\cdot(\nabla_{v}^\alpha A)\mathbf{h}_2}\lesssim\Big(\int_{\R^{d}}\mathbf{h}_1\cdot A\,\mathbf{h}_1\Big)^\frac12\Big(\int_{\R^{d}}\mathbf{h}_2\cdot A\,\mathbf{h}_2\Big)^\frac12
\end{equation}
\end{enumerate}
\end{lemma}
\begin{proof}
This lemma is Lemma 2.2 of \cite{duerinckx2023well}.
\end{proof}

\begin{lemma}\label{lemma:B_0 - properties}
(Properties of $\mathcal{B}_{0}$)
Let $V\in L ^1\cap \dot H^\frac12(\R^{d})$ be symmetric and positive definite.
\begin{enumerate}
\item {Boundedness:} For all $r\ge0$,
\begin{equation*}
\norm{\inn{v}^r\sqrt{\mu}\mathcal{B}_{0}[g]}_{ L^2_{v}}
\,\lesssim_{V,r}\, \norm{\inn{v}^{-r}g}_{ L ^2_{v}},
\end{equation*}
\item {Improved regularity:} Further assume $V\in\dot H^2(\R^{d})$ and $x V\in L ^2(\R^{d})$.
Then, for all~$\b >0$ and~$r\ge0$,
\begin{equation*}
\norm{\inn{v}^r\sqrt{\mu}\nabla_{v}^\b\mathcal{B}_{0}[g]}_{ L^2_{v}}
\,\lesssim_{V,\b,r}\,\sum_{\gamma<\b}\norm{\inn{v}^{-r}\nabla_{v}^\gamma g}_{ L^2_{v}}.
\end{equation*}
\end{enumerate}
\end{lemma}
\begin{proof}
This lemma is Lemma 2.3 of \cite{duerinckx2023well}.
\end{proof}

In view of the norm $\norm{\mathbf{h}}_{A}$ defined as in \eqref{weighted A norm definition}, using the coercivity and boundedness established in Lemma \ref{lemma:A - properties}, we note that any vector field $\mathbf{h}$, 
\begin{equation}\label{eq:equiv-L2A}
\norm{\mathbf{h}}_{A} \approx \norm{\inn{v}^{-\frac12}P_v^\perp \mathbf{h}}_{L^2_{v}}+\norm{\inn{v}^{-\frac32}P_v \mathbf{h}}_{L^2_{v}}.
\end{equation}

Recalling the decomposition \eqref{splitL} and combining Lemmas \ref{lemma:A - properties}-\ref{lemma:B_0 - properties}, we obtain the following corollary.

\begin{corollary}(Upper bound of $L$)\label{corollary:Upper bound of L}
Let $V\in L ^1(\R^d)\cap \dot H^\frac12(\R^{d})$ be symmetric, and let $\P$ be the orthogonal projection defined as in \eqref{Pf-abc-hat}. Then, there holds
\begin{equation}\label{upboundLg}
\begin{aligned}
\norm{\int_{\R^d} f L[g] \dd v} \lesssim \norm{\ip f}_{D} \norm{\ip g}_{D}.
\end{aligned}
\end{equation}

\end{corollary}
\begin{proof} Indeed, in view of the structure of $L$ in \eqref{splitL}, we compute 
\begin{align*}
\int_{\R^d} f L[g] \dd v 
=& - \int_{\R^d} \big((\nabla_{v} +v) f\big) \cdot \Big(A(\nabla_{v} +v)g -  \sqrt{\mu} \mathcal{B}_{0}[(\nabla_{v} + v) g]\Big)dd v
\end{align*}
Using Lemmas \ref{lemma:A - properties}-\ref{lemma:B_0 - properties} and the fact that $\int_{\R^d} f L[g] \dd v = \int_{\R^d} \ip f L[\ip g] \dd v$, we obtain the corollary.
\end{proof}

We next study the relationship between the norms we are considering. Recall that $\norm{\cdot}_{A} = \norm{\cdot}_{0,0,0,A}$ and $\norm{\cdot}_{D} = \norm{\cdot}_{0,0,0,D}$ defined as in Section \ref{sec-norm}. We obtain the following.

\begin{lemma}\label{lemma:A,D - norm}
Let $f$ be a function defined in $L^2(\mathbb{R}^d \times \mathbb{R}^d)$. Let $V\in L^1(\R^d)\cap\dot H^\frac12(\R^{d})$ be symmetric. Then, the following estimates hold for any $r \in \R$, $N,M \ge 0$, $p \in [1,\infty]$ and multi-index $\b$.
\begin{enumerate}
    \item For $A$-norm, 
    \begin{equation} \label{equivA}
\norm{\inn{v}^{-3/2} \mathbf{h}}_{L^2_{v}} \lesssim \norm{\mathbf{h}}_{A} \lesssim \norm{\inn{v}^{-1/2} \mathbf{h}}_{L^2_{v}}.
\end{equation}

\item For $D$-norm, 
\begin{equation} \label{equivD}
\begin{aligned}
\norm{\inn{v}^{-1/2}f}_{L^2_{v}} + \norm{\inn{v}^{-3/2}\nabla_v f}_{L^2_{v}} \lesssim \norm{f}_{D} \lesssim& \norm{\inn{v}^{-1/2}f}_{L^2_{v}} + \norm{\inn{v}^{-1/2}\nabla_{v} f}_{L^2_{v}}.
\end{aligned}
\end{equation}

\item For weighted $D$-norm,
\begin{equation} \label{ineq:weightedD}
\norm{f}_{N,M,r,D} \approx \norm{\inn{v}^{-r}f}_{N,M,0,D}.
\end{equation}

\item Let $\mathbf{P}f$ be the projection onto the kernel space of $L$ as in \eqref{Pf-abc-hat}, and let $a[f], b[f],$ and $c[f]$ be the projections onto the respective orthonormal bases as defined in \eqref{hat-abc-def}. Then, there hold 
\begin{align} \label{ineq:Pfchart}
\norm{\inn{v}^{-r}\pt_{\b}^{\a}\P f}_{L^{p}_{v}} \approx_{r,\b} \norm{a[\pt^{\a}f]} + \norm{b[\pt^{\a}f]}+ \norm{c[\pt^{\a}f]} = \norm{\pt^{\a}a[f]} + \norm{\pt^{\a}b[f]}+ \norm{\pt^{\a}c[f]},
\end{align}
for any $1 \le p \le \infty$.
\end{enumerate}

\end{lemma}
\begin{proof} The estimates in \eqref{equivA} follow directly from \eqref{eq:equiv-L2A}, upon recalling that $P_v, P_v^\perp$ are orthogonal projections onto $v$ and $v^\perp$, respectively, with $P_v + P_v^\perp = \mathbb{I}$. As for \eqref{equivD}, 
recalling the definition in \eqref{weighted D norm definition} and the estimate in \eqref{eq:equiv-L2A}, we have
\begin{align*}
\norm{f}^2_{D} =& \norm{v f }^2_{A} + \norm{\nabla_{v} f}^2_{A} 
\approx \norm{\langle v\rangle^{-3/2}vf}^2_{L^2_v} + \norm{\nabla_{v} f}^2_{A},
\end{align*}
upon noting that $P_v^\perp v =0$. This proves \eqref{equivD}, upon using \eqref{equivA}.  

Next we focus on \eqref{ineq:weightedD}.
Since $\norm{\cdot}$ is point-wise norm $x$ variable, it suffices to consider the case when $N=0$. By the definition of $\norm{\cdot}_{N,M,r,D}$ in \eqref{weighted D norm definition} and the properties of the $A$-norm in \eqref{eq:equiv-L2A}, we have
\begin{align*}
\norm{f}_{0,M,r,D}^2 =& \sum_{\norm{\b} \le M}\left(\int_{\R^d} \inn{v}^{-2r} \pt^{}_{\b} (vf) A \cdot \pt^{}_{\b} (vf) \dd v + \int_{\R^d} \inn{v}^{-2r} \pt^{}_{\b} (\nabla_v f) A \cdot \pt^{}_{\b} (\nabla_v f) \dd v\right)\\
\approx& \sum_{\norm{\b} \le M} \Bigg(\norm{\inn{v}^{-\frac12-r}P_v^\perp \pt^{}_{\b} (vf)}_{L^2_{v}}^{2} + \norm{\inn{v}^{-\frac12-r}P_v^\perp \pt^{}_{\b}(\nabla_v f)}_{L^2_{v}}^{2} \\
& \quad \qquad + \norm{\inn{v}^{-\frac32-r}P_v \pt^{}_{\b} (vf)}_{L^2_{v}}^{2} + \norm{\inn{v}^{-\frac32-r}P_v \pt^{}_{\b}(\nabla_v f)}_{L^2_{v}}^{2} \Bigg).
\end{align*}
On the other hand, we compute
\begin{align*}
\norm{\inn{v}^{-r}f}_{0,M,0,D}^2 
\approx& \sum_{\norm{\b} \le M} \Bigg(\norm{\inn{v}^{-\frac12}P_v^\perp \pt^{}_{\b} (v\inn{v}^{-r}f)}_{L^2_{v}}^{2} + \norm{\inn{v}^{-\frac12}P_v^\perp \pt^{}_{\b}(\nabla_v \inn{v}^{-r} f)}_{L^2_{v}}^{2} \\
& \quad \qquad + \norm{\inn{v}^{-\frac32}P_v \pt^{}_{\b} (v\inn{v}^{-r}f)}_{L^2_{v}}^{2} + \norm{\inn{v}^{-\frac32}P_v \pt^{}_{\b}(\nabla_v \inn{v}^{-r} f)}_{L^2_{v}}^{2} \Bigg).
\end{align*}
Note that $\norm{\pt_{\g} \inn{v}^{-r}} \lesssim \inn{v}^{-r}$ for any multi-index $\g$. Thus, for any $a,r\in \R$, we have
\begin{align*}
\norm{\inn{v}^{a}P_v^\perp \pt^{}_{\b} (v\inn{v}^{-r}f)}_{L^2_{v}}^{2} &\lesssim \sum_{\g \le \b} \norm{\inn{v}^{a}P_v^\perp \pt_{\g}{\inn{v}}^{-r} \pt^{}_{\b-\g} (vf)}_{L^2_{v}}^{2} \lesssim \sum_{\g \le \b} \norm{\inn{v}^{a-r}P_v^\perp \pt^{}_{\b-\g} (vf)}_{L^2_{v}}^{2}, \end{align*}
and 
$$
\norm{\inn{v}^{a}P_v \pt^{}_{\b} (v\inn{v}^{-r}f)}_{L^2_{v}}^{2}\lesssim \sum_{\g \le \b} \norm{\inn{v}^{a-r}P_v \pt^{}_{\b-\g} (vf)}_{L^2_{v}}^{2}.
$$
In addition, since $\nabla_{v} \inn{v}^{-r} = -rv \inn{v}^{-r-2}$, we have
\begin{align*}
\norm{\inn{v}^{a}P_v^\perp \pt^{}_{\b} (\nabla\inn{v}^{-r}f)}_{L^2_{v}}^{2} &\lesssim \sum_{\g \le \b} \left( \norm{\inn{v}^{a-r}P_v^\perp \pt^{}_{\b-\g} (\nabla f)}_{L^2_{v}}^{2} + \norm{\inn{v}^{a-r-2}P_v^\perp \pt^{}_{\b-\g} (v f)}_{L^2_{v}}^{2} \right), \\
\norm{\inn{v}^{a}P_v \pt^{}_{\b} (v\inn{v}^{-r}f)}_{L^2_{v}}^{2}&\lesssim \sum_{\g \le \b} \left( \norm{\inn{v}^{a-r}P_v \pt^{}_{\b-\g} (\nabla f)}_{L^2_{v}}^{2} + \norm{\inn{v}^{a-r-2}P_v \pt^{}_{\b-\g} (v f)}_{L^2_{v}}^{2} \right).
\end{align*}
Thus we have $\norm{\inn{v}^{-r}f}_{0,M,0,D} \lesssim \norm{f}_{0,M,r,D}$. If we set $g = \inn{v}^{-r}f$ and take $a$ in the above inequalities to be $-\frac{1}{2}-r$ or $-\frac{3}{2}-r$, we then have $f = \inn{v}^{r} g$ and so $ \norm{f}_{0,M,r,D} \lesssim \norm{\inn{v}^{-r}f}_{0,M,0,D}$. This proves \eqref{ineq:weightedD}. 

Finally, we focus on \eqref{ineq:Pfchart}. Since $\pt^\a_{x}$ commutes with the $L^p$ norm in $v$ and the projection operator, we only need to consider the case $\a = 0$. By the definition of $\mathbf{P}f$ in \eqref{Pf-abc-hat} and the fact that $a[f]$, $b[f]$, and $c[f]$ are functions depending only on $x$, we obtain the following identity:
\begin{align*}
\inn{v}^{-r}\pt_{v}^{\b}\P f = a[f]\inn{v}^{-r}\pt_{v}^{\b}\chi_{0} + \sum_{i=1}^{d}b_{i}[f] \inn{v}^{-r}\pt_{v}^{\b}\chi_{i} + c[f] \inn{v}^{-r}\pt_{v}^{\b} \chi_{d+1},
\end{align*}
Since each basis function $\chi_{i}$ takes the form of $\sqrt{\mu}$ multiplied by a polynomial in $v$, the function $\inn{v}^{-r}\partial_{v}^{\beta}\chi_{i}$ is also rapidly decaying and bounded in every $L^p$ norm, which yields \eqref{ineq:Pfchart}. This completes the proof of the lemma. 
\end{proof}

\subsection{Weighted Coercivity Estimates}

In this section, we establish the coercive structure of the linearized operator $L$ in weighted spaces. Precisely, we obtain the following lemma.

\begin{lemma} \label{lemma:Coerceivity of L} (Coercivity of $L$)
Let $V\in L^1(\R^d)\cap\dot H^\frac12(\R^{d})$ be symmetric and positive definite. 
\begin{enumerate}
\item Semi-coercivity of operator $L$ : For any $g(x,\cdot) \in L^2(\R^d_{v})$
\begin{equation}\label{coerL}
-\int_{\R^d} \pt^{\a} g \pt^{\a}L[g] \dd v \gtrsim \norm{\ip \pt^{\a} g}_{D}^2.
\end{equation}

\item Weighted coercivity of operator $\pt_{\b} L$ : For any multi-index $\b >0$, $r \in \mathbb{R}$, and $g(x,\cdot) \in L^2(\R^d_{v})$
\begin{equation}\label{weightLg1}
\begin{aligned}
-\int_{\R^d} \inn{v}^{-2r} \pt_{\b}g \pt_{\b}L[g] \dd v \ge \eta \norm{\inn{v}^{-r}\pt_{\b} g}_{D}^2
- C_{\eta,\b} \sum_{\norm{\b'} \le \norm{\b}-1}\norm{\inn{v}^{-r} \pt_{\b'} g}_{D}^2.
\end{aligned}
\end{equation}
If $\b =0$, we have
\begin{equation}\label{weightLg2}
\begin{aligned}
-\int_{\R^d} \inn{v}^{-2r} g L[g] \dd v \ge \eta\norm{\inn{v}^{-r} g}_{D}^2
- C_{\eta}\norm{g}_{D}^2,
\end{aligned}
\end{equation}
for some $\eta>0$.
\end{enumerate}
\end{lemma}
\begin{proof}

The proof of semi-coercivity can be found in Step 2 of the proof of Theorem 1 in \cite{duerinckx2023well}. We therefore focus on establishing the weighted coercivity estimates. 

\subsubsection*{Proof of \eqref{weightLg1}.} 

For $\b>0$, we compute 
\begin{equation}\label{compLg}
\begin{aligned}
\int_{\R^d} \inn{v}^{-2r} \pt_{\b}g \pt_{\b}L[g] \dd v &= \int_{\R^d} \pt_{\b}[\inn{v}^{-r} g] L[\inn{v}^{-r}\pt_{\b}g] \dd v +  \int_{\R^d} [\inn{v}^{-r} ,\pt_{\b}]g L[\inn{v}^{-r}\pt_{\b}g] \dd v 
\\&\quad + \int_{\R^d} \inn{v}^{-r} (\pt_{\b} g) [\inn{v}^{-r} \pt_{\b},L]g \dd v 
\end{aligned}
\end{equation}
in which $[A,B] = AB - BA$ denotes the usual commutator. Using \eqref{coerL}, we bound 
\begin{equation}\label{lowbdLg}
\begin{aligned}
-\int_{\R^d} \pt_{\b}[\inn{v}^{-r} g] L[\inn{v}^{-r}\pt_{\b}g] \dd v 
&\gtrsim \norm{\ip (\inn{v}^{-r}\pt_{\b}g)}_{D}^2
\\
&\gtrsim \norm{\inn{v}^{-r}\pt_{\b}g}_{D}^2 - \norm{\inn{v}^{-r}\P g}_{D}^2 ,
\end{aligned}
\end{equation}
in which, using \eqref{ineq:Pfchart} and recalling \eqref{hat-abc-def}, we may bound 
$$
\begin{aligned}
\norm{\inn{v}^{-r}\P g}_{D} 
&\lesssim 
\norm{a[g]} + \norm{b[g]}+ \norm{c[g]} \lesssim \norm{\inn{v}^{-r-1/2}g}_{L^2_v} \lesssim \norm{\inn{v}^{-r}g}_{D} .
\end{aligned}$$
It remains to bound the last two integrals in \eqref{compLg}. Indeed, using the upper bound \eqref{upboundLg}, we obtain 
$$\Big| \int_{\R^d} [\inn{v}^{-r} ,\pt_{\b}]g L[\inn{v}^{-r}\pt_{\b}g] \dd v \Big| \lesssim \norm{\inn{v}^{-r}\pt_{\b}g}_{D} \norm{[\inn{v}^{-r} ,\pt_{\b}]g}_D
, $$
in which we may bound $|[\inn{v}^{-r} ,\pt_{\b}]g|\lesssim \sum_{\norm{\b'} < \norm{\b}}\norm{\inn{v}^{-r} \pt_{\b'}g}$. As for the last integral term, we claim that for all $\b >0$,
\begin{align}\label{ineq:lemma7-step1}
\norm{\int_{\R^d} \inn{v}^{-r} (\pt_{\b} g) [\inn{v}^{-r} \pt_{\b},L]g \dd v} \lesssim_{V,\b} \norm{\inn{v}^{-r}\pt_{\b} g}_{D} \sum_{\norm{\b'} < \norm{\b}}\norm{\inn{v}^{-r} \pt_{\b'}g}_{D}.
\end{align}
The estimate \eqref{weightLg1} would then follow, upon using the standard Young's inequality and the lower bound \eqref{lowbdLg}. To prove the claim \eqref{ineq:lemma7-step1},  recalling \eqref{splitL}, we decompose
\begin{align*}
\int_{\R^d} \inn{v}^{-r} (\pt_{\b} g) [\inn{v}^{-r} \pt_{\b},L]g \dd v = T^{r,\b}_{1} - T^{r,\b}_{2},
\end{align*}
in which 
\begin{align*}
T^{r,\b}_{1} :=& \int_{\R^d} \inn{v}^{-r} (\pt_{\b} g) [\inn{v}^{-r} \pt_{\b},(\nabla_{v} - v) \cdot (A(\nabla_{v}+v))]g \dd v,\\
T^{r,\b}_{2} :=& \int_{\R^d} \inn{v}^{-r} (\pt_{\b} g) [\inn{v}^{-r} \pt_{\b}, (\nabla_{v} - v) \cdot (\sqrt{\mu} \mathcal{B}_{0}(\nabla_{v}+v))]g \dd v.
\end{align*}
A direct computation yields
\begin{align*}
&[\inn{v}^{-r} \pt_{\b},\nabla_{v}] = -r v\inn{v}^{-(r+2)}\pt_{\b}, \quad [\inn{v}^{-r} \pt_{\b},v] = \inn{v}^{-r} \sum_{\substack{e_{j} \le \b \\ \norm{e_{j}}=1}}e_{j}\pt_{\b-e_{j}},\\
&[\inn{v}^{-r}\pt_{\b},A] = \sum_{\substack{\g \le \b \\ \g \neq 0}} \binom{\b}{\g} \inn{v}^{-r} \pt_{\g}A \pt_{\b-\g}.
\end{align*}
Therefore, we compute 
\begin{align*}
& [\inn{v}^{-r} \pt_{\b},(\nabla_{v} - v) \cdot (A(\nabla_{v}+v))] \\
&= -r v \inn{v}^{-r-2}\cdot\sum_{\g \le \b} \binom{\b}{\g} \pt_{\g}A\bigg( {(\nabla_{v} +v) \pt_{\b-\g}} +  {\sum_{\substack{e_{j} \le \b-\g \\ \norm{e_{j}}=1}} e_{j}\pt_{\b-\g-e_{j}}} \bigg) \\
&\quad+ {\inn{v}^{-r}\sum_{\substack{e_{j} \le \b \\ \norm{e_{j}}=1}}e_{j} \cdot \sum_{\g \le \b-e_{j}} \binom{\b-e_{j}}{\g}\pt_{\g}A\bigg((\nabla_{v} +v) \pt_{\b-\g-e_{j}} + \sum_{\substack{e_{k} \le \b-\g-e_{j} \\ \norm{e_{k}}=1}} e_{k}\pt_{\b-\g-e_{j}-e_{k}} \bigg)} \\
&\quad+ {(\nabla_{v}-v) \cdot \sum_{\substack{\g \le \b \\ \g \neq 0}} \binom{\b}{\g} \inn{v}^{-r} \pt_{\g}A \bigg((\nabla_{v}+v)\pt_{\b-\g}  + \sum_{\substack{e_{j} \le \b-\g \\ \norm{e_{j}}=1}} e_{j}\pt_{\b-\g-e_{j}} \bigg)} \\
&\quad+ (\nabla_{v} - v) \cdot A \bigg( {- r v \inn{v}^{-r-2} \pt_{\b}}  +  {\inn{v}^{-r} \sum_{\substack{e_{j} \le \b \\ \norm{e_{j}}=1}}e_{j}\pt_{\b-e_{j}}} \bigg).
\end{align*}
Let us first focus on $T^{r,\b}_{1}$. After shifting $(\nabla_{v}-v)$ from the commutator to the $\partial_\b g$ term via integration by parts, we apply Lemma \ref{lemma:A - properties}. A careful count of the derivatives reveals that no terms of the form $\nabla_{v} \partial_{\b} g \cdot A \nabla_{v} \partial_{\b} g$ exist; instead, every term lacks at least one derivative. By further utilizing $\langle v \rangle^{-1} \le 1$ and the commutativity of the weight with $A$, we obtain
\begin{align*}
\norm{T^{r,\b}_{1}} \lesssim \sum_{\g < \b}\norm{\inn{v}^{-r}\pt_{\b}g}_{D} \norm{\inn{v}^{-r}\pt_{\g}g}_{D}.
\end{align*}
On the other hand, for $T^{r,\b}_{2}$, we further write 
\begin{align*}
T^{r,\b}_{2} =& \int_{\R^d}\inn{v}^{-2r} \pt_{\b} g\pt_{\b} \Big((\nabla_{v}-v)\cdot (\sqrt{\mu}\mathcal{B}_{0}[(\nabla_{v}+v)g]) \Big) \dd v\\
&- \int_{\R^d}\inn{v}^{-r} \pt_{\b}g \Big((\nabla_{v}-v)\cdot (\sqrt{\mu}\mathcal{B}_{0}[(\nabla_{v}+v)\inn{v}^{-r}\pt_{\b}g]) \Big) \dd v\\
=& T^{r,\b}_{2,1} + T^{r,\b}_{2,2},
\end{align*}
where $T^{r,\beta}_{2,1}$ can be written as follows, upon taking integration by parts in $v$, 
\begin{align*}
T^{r,\b}_{2,1} = &- \int_{\R^d}(\nabla_{v}+v)(\inn{v}^{-r} \pt_{\b}g) \cdot \inn{v}^{-r}\pt_{\b} \Big( (\sqrt{\mu}\mathcal{B}_{0}[(\nabla_{v}+v)g]) \Big) \dd v\\
& - \sum_{\substack{e_{j} \le \b \\ \norm{e_{j}} = 1 }}\int_{\R^d}\inn{v}^{-r} \pt_{\b}g \cdot e_{j} \inn{v}^{-r}\pt_{\b} \Big( (\sqrt{\mu}\mathcal{B}_{0}[(\nabla_{v}+v)g]) \Big) \dd v \\
& +r \int_{\R^d}v(\inn{v}^{-r} \pt_{\b}g) \cdot \inn{v}^{-r-2}\pt_{\b} \Big( (\sqrt{\mu}\mathcal{B}_{0}[(\nabla_{v}+v)g]) \Big) \dd v.
\end{align*}
Applying Lemma \ref{lemma:B_0 - properties} and Lemma \ref{lemma:A,D - norm}, we obtain
\begin{align*}
\norm{T^{r,\b}_{2,1}} \lesssim \sum_{\norm{\b'} < \norm{\b}}\norm{\inn{v}^{-r}\pt_{\b}g}_{D} \norm{\inn{v}^{-r}\pt_{\b'}g}_{D}.
\end{align*}
Similarly, for $T^{r,\beta}_{2,2}$, we write upon integrating by parts for some $e_{j} \le \b$, 
\begin{align*}
T^{r,\beta}_{2,2} =& - \int_{\R^d}(\nabla_{v}+v) (\inn{v}^{-r} \pt_{\b-e_{j}}g) \cdot \pt_{e_{j}}\Big( (\sqrt{\mu}\mathcal{B}_{0}[(\nabla_{v}+v)\inn{v}^{-r}\pt_{\b}g]) \Big) \dd v \\
&- \int_{\R^d}(e_{j} -r vv_{j}\inn{v}^{-2})(\inn{v}^{-r} \pt_{\b-e_{j}}g) \cdot \Big( (\sqrt{\mu}\mathcal{B}_{0}[(\nabla_{v}+v)\inn{v}^{-r}\pt_{\b}g]) \Big) \dd v
\end{align*}
Applying Lemma \ref{lemma:B_0 - properties} and Lemma \ref{lemma:A,D - norm}, we obtain
\begin{align*}
\norm{T^{r,\b}_{2,2}} \lesssim \sum_{\norm{\b'} < \norm{\b}}\norm{\inn{v}^{-r}\pt_{\b}g}_{D} \norm{\inn{v}^{-r}\pt_{\b'}g}_{D}.
\end{align*}
This proves \eqref{ineq:lemma7-step1}, and therefore, the weighted coercivity estimate \eqref{weightLg1} as claimed.

\subsubsection*{Proof of \eqref{weightLg2}.} 

Finally, we focus on the case when $\b=0$, namely proving the weighted coercivity estimate \eqref{weightLg2}. First, consider $r\ge 0$. Note that $\norm{g}_{D} \ge \norm{\inn{v}^{-r}g}_{D}$. Therefore, using the upper bound \eqref{upboundLg}, we obtain 
\begin{align*}
-\int_{\R^d} \inn{v}^{-2r} g L[g] \dd v  \ge -C \norm{g}_{D}^2 \ge \norm{\inn{v}^{-r} g}_{D}^2
- (C+1) \norm{g}_{D}^2.
\end{align*}
Next, we consider $r <0$. In view of \eqref{splitL}, we compute 
\begin{align*}
-\int_{\R^d} \inn{v}^{-2r} g L[g] \dd v =& -\int_{\R^{d}} \inn{v}^{-2r} g \Big((\nabla_{v} -v) \cdot A (\nabla_{v} +v) g \\
&- (\nabla_{v} - v) \cdot (\sqrt{\mu} \mathcal{B}_{0}[(\nabla_{v} + v)g]) \Big) \dd v \\ 
=& \int_{\R^{d}} (\nabla_{v} +v) (\inn{v}^{-r} g) \cdot A (\nabla_{v} +v) (\inn{v}^{-r} g) \dd v \\
&-\int_{\R^{d}} \inn{v}^{-r} g [\inn{v}^{-r} ,(\nabla_{v} - v) \cdot (A(\nabla_{v}+v))]g \dd v \\
& -\int_{\R^{d}} \inn{v}^{-2r} g (\nabla_{v} - v) \cdot (\sqrt{\mu} \mathcal{B}_{0}[(\nabla_{v} + v)g] )] \dd v \\
=& T^{r}_{1} + T^{r}_{2} + T^{r}_{3}.
\end{align*}
Thanks to the structure of $A$ and the definition of the $A$-norm, we have
\begin{align*}
T^{r}_{1} \ge  \eta\norm{\inn{v}^{-r} g}_{D}^2,
\end{align*}
for some $\eta>0$. Following the same procedure as done in the previous step, we have
\begin{align*}
\norm{T^{r}_{2}} \le& \int_{\R^d} \norm{rv\inn{v}^{-r}g \cdot A(\nabla_{v}+v)(\inn{v}^{-(r+2)}g)} \dd v \\
&+ \int_{\R^d} \norm{r(r+2)v\inn{v}^{-r}g \cdot Av\inn{v}^{-(r+4)}g} \dd v \\
&+ \int_{\R^d} \norm{r(\nabla_{v}+v)(\inn{v}^{-r}g) \cdot Av\inn{v}^{-(r+2)}g} \dd v.
\end{align*}
Take any $M>0$. Recalling that $r<0$, we bound 
\begin{align*}
\int_{\R^d} \norm{rv\inn{v}^{-r}g \cdot A(\nabla_{v}+v)(\inn{v}^{-(r+2)}g)} \dd v \le& \int_{\norm{v} \le M} \norm{rv\inn{v}^{-r}g \cdot A(\nabla_{v}+v)(\inn{v}^{-(r+2)}g)} \dd v \\
&+ \int_{\norm{v} > M} \norm{rv\inn{v}^{-r}g \cdot A(\nabla_{v}+v)(\inn{v}^{-(r+2)}g)} \dd v \\
\lesssim & \langle M\rangle^{2|r|}\int_{\norm{v} \le M} \norm{vg \cdot A(\norm{\nabla_{v}g}) + \norm{{v}g})} \dd v \\
&+ \langle M\rangle^{-2} \int_{\norm{v} > M} \norm{v\inn{v}^{-r}g \cdot A(\nabla_{v}+v)(\inn{v}^{-r}g)} \dd v \\
\lesssim & \langle M\rangle^{2|r|} \norm{g}_{D}^2 + \langle M\rangle^{-2} \norm{\inn{v}^{-r}g}_{D}^2
\end{align*}
The other terms can be bounded in a similar manner. Therefore, taking $M$ sufficiently large so that $\langle M\rangle^{-2} \ll \eta$, 
we obtain 
\begin{align*}
\norm{T^{r}_{2}} \le  C_\eta\norm{g}_{D}^2 + \frac{\eta}{2} \norm{\inn{v}^{-r}g}_{D}^2.
\end{align*}
Finally, for $T^{r}_{3}$, using Lemma \ref{lemma:B_0 - properties}, we bound 
\begin{align*}
\norm{T^{r}_{3}} \le& \norm{\int_{\R^{d}} \inn{v}^{-2r} g (\nabla_{v} - v) \cdot (\sqrt{\mu} \mathcal{B}_{0}[(\nabla_{v} + v)g] )] \dd v} \\
\le & \norm{\int_{\R^{d}}  (\nabla_{v} + v(1-2r \inn{v}^{-2}))g  \cdot \inn{v}^{2|r|}(\sqrt{\mu} \mathcal{B}_{0}[(\nabla_{v} + v)g] )] \dd v} \\
\lesssim & \norm{g}_{D}^2.
\end{align*}
This completes the proof of \eqref{weightLg2}, and hence, the lemma. 
\end{proof}

\section{Nonlinear Collision Operator}\label{sec-nonop}

In this section, we shall estimate the nonlinear collision terms $ N(g_{1},g_{2},g_{3})$ in weighted energy norms. Precisely, we obtain the following proposition. 

\begin{proposition} \label{proposition:N upper bound} Let $N(g_1,g_2,g_3)$ be the nonlinear term defined as in \eqref{defNg123}. Assume $V \in L^1(\R^d) \cap \dot{H}^{{2}}(\R^d)$ be symmetric and positive definite so that $x V \in L^2(\R^d)$. In addition, assume that $g_{1}(x,\cdot)\in L^2(\R^{d})$ satisfies the following smallness condition, for some $r_0\ge0$, $\delta_0>0$, and some large enough constant $C_0$,
\begin{equation}
    \norm{\inn{v}^{-r_{0}} \inn{\nabla_{v}}^{\frac{3}{2}+\d_{0}}g_{1}(x,\cdot)}_{L^2_{v} } \le \frac{1}{C_0}.
\end{equation}
Then, for any $r \ge 0$, $\theta \in \R$, and any multi-index $\a$ and $\b$, there holds
 \begin{equation} \label{cor 7 - inequality}
\begin{aligned}
&\norm{\iint_{\T^d \times \R^d} \inn{v}^{-2\th} h\pt^{\a}_{\b}N(g_{1},g_{2},g_{3}) \;\dd x\dd v } 
\\
&\lesssim \normm{h}_{0,0,\th,D} \sum_{} \Big(\normm{g_{3}}_{\norm{\sigma_{3}},\norm{\g_{3}},\th,D} \normm{g_{2}}_{\norm{\sigma_{2}},\norm{\g_{2}},r,D} \prod_{j}(1+\normm{g_{1}}_{\norm{\sigma_{1,j}},\norm{\g_{1,j}},r,D}),
\\
&\quad+\normm{g_{3}}_{\norm{\sigma_{3}},\norm{\g_{3}},\th,D} \normm{g_{1}}_{\norm{\sigma_{1,1}},\norm{\g_{1,1}},r,D}  \prod_{j}(1+\normm{g_{1}}_{\norm{\sigma_{1,j}},\norm{\g_{1,j}},r,D}) \\
&\quad+ \normm{g_{2}}_{\norm{\sigma_{2}},\norm{\g_{2}},r,D} \normm{g_{1}}_{\norm{\sigma_{1,1}},\norm{\g_{1,1}},r,D} \prod_{j}(1+\normm{g_{1}}_{\norm{\sigma_{1,j}},\norm{\g_{1,j}},r,D})\Big)
\end{aligned}
\end{equation} 
for any functions $h$ and $g_j$ such that the right-hand side is finite. Here, the summation is taken over all partitions of the multi-indices $\a$ and $\b$ satisfying 
\begin{equation}\label{partitionab}\sum_{j}\a_{1,j} + {\a_{2}} + {\a_{3}} = {\a},\qquad \sum_{j}\b_{1,j} + {\b_{2}} + {\b_{3}} = {\b},
\end{equation}
and the pairs $(\sigma_{i},\gamma_{i})$ are defined as follows: 

\begin{itemize}

\item For any $i \in \{(1,1),\cdots (1,n),2\}$, if $(|\a_i|,|\b_i|)\not = (|\a|,|\b|)$, then set
\begin{equation}\label{defsigmai1}
(\sigma_i, \gamma_i) = \left\{ \begin{aligned}
(|\alpha_i|+\lfloor\frac{d}{2}+1 \rfloor , |\beta_i|+1) , &\qquad \mbox{if} \quad\norm{\a_i}+\norm{\b_i} < \frac{d}{2}+5,
\\
(|\alpha_i|, |\beta_i|+1) , &\qquad \mbox{if} \quad\norm{\a_i}+\norm{\b_i} \ge \frac{d}{2}+5.
\end{aligned}
\right.
\end{equation}
Otherwise, we set $(\sigma_i, \gamma_i) = (|\a|,|\b|)$ if $(|\a_i|,|\b_i|) = (|\a|,|\b|)$. 

\item For $i=3$, if $(\a_3,\b_3) =0$, set $(\sigma_3,\gamma_3) = (0,1)$, and
if $(\a_3,\b_3)\not =0$, then set 
\begin{equation}\label{defsigmai}
(\sigma_3, \gamma_3) = \left\{ \begin{aligned}
(|\alpha_3|, |\beta_3|) , &\qquad \mbox{if} \quad
\norm{\a_i}+\norm{\b_i} < \frac{d}{2}+5, \quad \forall ~i\in \{(1,1),\cdots (1,n),2\}, 
\\
(|\alpha_3|, |\beta_3|) , &\qquad \mbox{if} \quad\norm{\a_3}+\norm{\b_3} \ge \frac{d}{2}+5,
\\
(|\alpha_3|+\lfloor\frac{d}{2}+1 \rfloor , |\beta_3|) , &\qquad \mbox{if otherwise.} \end{aligned}
\right.
\end{equation}

\end{itemize}
Here, $\lfloor a\rfloor$ denotes the greatest integer that is less than or equal to $a$.  

\end{proposition}

Proposition \ref{proposition:N upper bound} provide uniform upper bounds on the nonlinear terms that will be sufficient to close the nonlinear iterative scheme developed in the next section. Let us comment on the apparance of the multi-indexes $(\a_i,\b_i)$ and the pairs $(\sigma_i,\gamma_i)$. First, the partitions \eqref{partitionab} are due to the standard Leibniz's rule for derivatives of a product involving three functions $g_1,g_2,g_3$ in $N(g_1,g_2,g_3)$, see \eqref{defNg123}, while $(\alpha_{1,j},\beta_{1,j})$ are due to derivatives of a composite function $ \e (k,k\cdot v; \nabla_{v} F_{g_1})$. As for the corresponding pairs $(\sigma_{i},\gamma_{i})$, either they are $(|\alpha_i|,|\beta_i|+1)$ (i.e. allowing one loss of $v$-derivatives unless they are already top derivatives), or $(|\alpha_i|+\lfloor\frac{d}{2}+1 \rfloor , |\beta_i|+1)$, where the loss of $x$-derivatives is due to the Sobolev embedding $H_x^{\frac{d}{2}+\delta} \subset L^\infty_x$, which were used precisely for the low norm with $\norm{\a_i}+\norm{\b_i} < \frac{d}{2}+5$. Similar definitions apply for $(\sigma_3,\gamma_3)$, except we do not allow the loss of $v$-derivatives, unless when $(\alpha_3,\beta_3)=0$. Finally, we note that we gain an arbitrary decay in $v$ for $g_1$ and $g_2$ thanks to the rapid decay of the Maxwellian $\mu(v)$. The rest of this section is devoted to prove Proposition \ref{proposition:N upper bound}.

\begin{remark}\label{rem-abi}
We note that if $\norm{\a_i}+\norm{\b_i} < \frac{d}{2}+5$, then the corresponding pair $(\sigma_i,\gamma_i)$ satisfies 
\begin{equation}\label{goodsg}
\norm{\sigma_{i}} +\norm{\g_{i}} \le d+6,
\end{equation}
since by construction, $\norm{\sigma_{i}} +\norm{\g_{i}} \le |\alpha_i|+\lfloor\frac{d}{2}+1 \rfloor + |\beta_i|+1 < d+7$ (and $\norm{\sigma_{i}},\norm{\g_{i}}$ are integers). On the other hand, if $\norm{\a_i}+\norm{\b_i} \ge \frac{d}{2}+5$, then $\norm{\sigma_{i}} +\norm{\g_{i}} \le |\alpha|+|\beta|$, since either $(\sigma_i,\gamma_i) = (|\alpha|,|\beta|)$ if $(|\a_i|,|\b_i|) = (|\a|,|\b|)$ or else $\norm{\sigma_{i}} +\norm{\g_{i}} \le |\alpha_i|+|\beta_i|+1 \le |\alpha|+|\beta|$. 

\end{remark}

\subsection{Linear Operator $\mathcal{B}_{\nabla_{v} F}[h]$}

We first study the linear operator $\mathcal{B}_{\nabla_{v} F}[h]$. Indeed, recalling \eqref{kernelB} and \eqref{defcB}, and applying Leibniz's rule, we can decompose the expression as follows:

\begin{equation} \label{B leibniz rule}
\pt_{\b}^{\a} \mathcal{B}_{\nabla_{v} F_{g}}[h] = \sum_{\substack{\a_{1} + \a_{2} = \a\\ \b_{1} + \b_{2} = \b}} \binom{\a}{\a_{1}} \binom{\b}{\b_{1}} G^{\a_1,\b_1,\a_{2},\b_{2}} (g,h)
\end{equation}
where 
\begin{equation}\label{defGOG}
\begin{aligned}
G^{\a_{1},\b_{1},\a_{2},\b_{2}} (g,h) := \iint_{\R^d \times \R^d} (k \otimes k) |\widehat{V}(k)|^2 \d (k \cdot (v-v_{*})) \pt_{\b_{1}}^{\a_{1}} \left( \frac{1}{\norm{\e(k,k\cdot v ; \nabla_{v} F_{g})}^2}\right) \\
\times \pt_{\a_{2}}^{\b_{2}} (\sqrt{\mu_{*}}h_{*}) \dd v_{*} \dd k.
\end{aligned}
\end{equation}
By Fa\`a di Bruno's formula, the derivatives of $|\e(k,k\cdot v;\nabla_{v} F_g)|^{-2}$ are computed by 
\begin{equation}
\begin{aligned}
{\pt_{\b_{1}}^{\a_{1}} \left( \frac{1}{\norm{\e(k,k\cdot v ; \nabla_{v} F_{g})}^2}\right)}
&=\sum_{\substack{ \sum_{i=1}^{s} \eta_{1,1,i} + \sum_{j=1}^{l} \eta_{1,2,j} = \a_{1}\\ \sum_{i=1}^{s} \eta_{2,1,i} + \sum_{j=1}^{l} \eta_{2,2,j} = \b_{1}}}\frac{C_{\eta}}{{\e}^{s+1} \bar{\e}^{l+1}}
\prod_{i=1}^{s} \prod_{j=1}^{l} ({\pt^{\eta_{1,1,i}}_{\eta_{2,1,i}} \e})(\overline{{\pt^{\eta_{1,2,j}}_{\eta_{2,2,j}} \e}}),
\end{aligned}
\end{equation}
for $\e = \e(k,k\cdot v; \nabla_{v}F_{g})$, in which we recall that $\pt^\eta = \pt_x^\eta$ and $\pt_\eta = \pt_v^\eta$. Since the summation derived from Fa\'a di Bruno's formula is too lengthy, we will omit the conditions under the summation sign in such expressions where no confusion arises. For convenience, we set $\Omega_{1} = (\a_{1},\b_{1})$, $\Omega_{2} = (\a_{2},\b_{2})$, and let $\Gamma = \{(\eta_{1,1,1},\eta_{2,1,1}),(\eta_{1,2,1},\eta_{2,2,1}),\cdots \}$ be the collection of all multi-indices. Define
\begin{equation}\label{expGOmega}
G^{\Omega_{1},\Omega_{2}} (g,h) = \sum_{\Gamma} C_{\Omega_{1},\Omega_{2},\Gamma} T^{\Omega_{1},\Omega_{2}}_{\G}(g,h),
\end{equation}
where 
\begin{equation} \label{TAGpi - definition}
\begin{aligned}
T^{\Omega_{1},\Omega_{2}}_{\G}(g,h) :&= \iint_{\R^d \times \R^d} (k \otimes k) |\widehat{V}(k)|^2 \d (k \cdot (v-v_{*})) \\
&\qquad\times \frac{C_{\eta}}{{\e}^{s+1} \bar{\e}^{l+1}}
\prod_{i=1}^{s} \prod_{j=1}^{l} ({\pt^{\eta_{1,1,i}}_{\eta_{2,1,i}} \e})(\overline{{\pt^{\eta_{1,2,j}}_{\eta_{2,2,j}} \e}})
\pt_{\a_{2}}^{\b_{2}} (\sqrt{\mu_{*}}h_{*}) \dd v_{*} \dd k,
\end{aligned}
\end{equation}
for $\e = \e(k,k\cdot v; \nabla_{v}F_{g})$. 
We obtain the following. 

\begin{lemma} \label{lemma:T - upper bound}
(Upper bound of $T^{\Omega_{1},\Omega_{2}}_{\G}$). Let $V \in L^1 \cap \dot{H}^{\frac{1}{2}}$ be symmetric and positive definite. Given $g(x,\cdot),h(x,\cdot) \in L^2_{loc}$, provided $g$ satisfies the following smallness condition, for some $r_0 \ge 0$, $\d >0$, and some large enough constant $C_{0}$,
\begin{equation}\label{smallnessg1}
    \norm{\inn{v}^{-r_{0}} \inn{\nabla_{v}}^{\frac{3}{2}+\d} g(x,\cdot)}_{L^2} \le \frac{1}{C_{0}},
\end{equation}
then, we have for all vector fields $\mathbf{h}_{1}$, $\mathbf{h}_{2}$, for all multi-index $\Omega_{1}$, $\Omega_{2}$, $\G$ and real number $r \ge 0$, and $\th \in \R$
\begin{equation}\label{T - upper 1}
\begin{aligned}
&\norm{\int_{\R^d} \inn{v}^{-2\th} \mathbf{h}_{1} \cdot T^{\Omega_{1},\Omega_{2}}_{\G}(g,h) \mathbf{h}_{2} \dd v} \\
&\lesssim 
\norm{\mathbf{h}_{1}}_{0,0,\th,A} \norm{\mathbf{h}_{2}}_{0,0,\th,A} \norm{h}_{\norm{\a_{2}},\norm{\b_{2}}+1,r,L^2} \prod_{(\eta_{1,i,j},\eta_{2,i,j}) \in \G}(1+\norm{g}_{\norm{\eta_{1,i,j}},\norm{\eta_{2,i,j}}+2,r,L^2}).
\end{aligned}
\end{equation}
Alternatively, we can exchange one derivative of $h$ with one derivative of $\mathbf{h}_{2}$, namely
\begin{equation}\label{T - upper 2}
\begin{aligned}
&\norm{\int_{\R^d}\inn{v}^{-2\th} \mathbf{h}_{1} \cdot T^{\Omega_{1},\Omega_{2}}_{\G}(g,h) \mathbf{h}_{2} \dd v} \\
&\lesssim 
\norm{\mathbf{h}_{1}}_{0,0,\th,A} \norm{\mathbf{h}_{2}}_{0,1,\th,A} \norm{h}_{\norm{\a_{2}},\norm{\b_{2}},r,L^2} \prod_{(\eta_{1,i,j},\eta_{2,i,j}) \in \G}(1+\norm{g}_{\norm{\eta_{1,i,j}},\norm{\eta_{2,i,j}}+2,r,L^2}).
\end{aligned}
\end{equation}
In particular, when all the derivatives hit $\epsilon$ or $\bar{\epsilon}$, namely $(s,l)=(1,0)$ or $(s,l)=(0,1)$ and $\eta_{1,1,1} = \a, ~ \eta_{2,1,1} = \b$ or $\eta_{1,2,1} = \a,~  \eta_{2,2,1} = \b$, we may exchange one derivative of $h$ with one derivative of $g$, namely
\begin{equation}\label{T - upper 3}
\begin{aligned}
\norm{\int_{\R^d}\inn{v}^{-2\th} \mathbf{h}_{1} \cdot T^{\Omega_{1},\Omega_{2}}_{\G}(g,h) \mathbf{h}_{2} \dd v} \lesssim 
\norm{\mathbf{h}_{1}}_{0,0,\th,A} \norm{\mathbf{h}_{2}}_{0,1,\th,A} 
\norm{h}_{0,1,r,L^2} \norm{g}_{\norm{\a},\norm{\b}+1,r,L^2}.
\end{aligned}
\end{equation}
\end{lemma}

In the norms appearing in the lemma, it is important to note that only the number of $v$-derivatives increases, while the number of $x$-derivatives remains unchanged. This is because our lemma is based on pointwise estimates in $x$, and thus it does not affect the structural properties related to $x$. Another crucial point is that, through this lemma, we can establish the $A$-norm for $\mathbf{h}_1$ and $\mathbf{h}_2$ with respect to the weights prescribed in the equation, while for $g$ and $h$, we obtain a decay of $\langle v \rangle^{-r}$ for any positive $r$. The sufficient decay of $g$ and $h$ stems from the structural fact that they are multiplied by the Maxwellian.

Furthermore, the structure of the three inequalities presented in the lemma is as follows. Examining how much the velocity derivatives of $\mathbf{h}_2$, $h$, and $g$ increase within the norm reveals that: in the \eqref{T - upper 1}, the derivative of $\mathbf{h}_2$ remains unchanged, while the derivatives of $h$ and $g$ increase by one and two, respectively. In the \eqref{T - upper 2}, the derivative of $h$ remains unchanged, while the derivatives of $\mathbf{h}_2$ and $g$ increase by one and two, respectively. Finally, in the \eqref{T - upper 3}, the derivative of $g$ increases by one, and the derivatives of $\mathbf{h}_2$ and $h$ each increase by one. Each inequality is applied when the highest derivative hits $\mathbf{h}_2$, $h$, and $g$, respectively.
In addition, when $\mathbf{h}_1$ and $\mathbf{h}_2$ are given as vector fields in the form of $v f$ (as seen in \eqref{defNg123}, where they are defined via $\nabla f$ and $v f$), the presence of $k \otimes k$ and $\delta(k \cdot (v-v_{*}))$ in the collision kernel $B$ allows us to exchange $v$ for $v_{*}$ during the integration process. This implies that the weight can be transferred to $h$, thereby yielding an additional decay of $\inn{v}$ for vector fields given in the form $v f$.

\begin{proof}
[Proof of Lemma \ref{lemma:T - upper bound}]
First, we will prove \eqref{T - upper 1}.
Recalling $\e = \e(k,k\cdot v; \nabla_{v}F_{g})$ and using Lemma \ref{lem-disp}, precisely \eqref{epsilon - space and velocity derivative - estimate}, we bound 
\begin{equation*}
\begin{aligned}
&\norm{\prod_{i=1}^{s} \prod_{j=1}^{l} ({\pt^{\eta_{1,1,i}}_{\eta_{2,1,i}} \e})(\overline{{\pt^{\eta_{1,2,j}}_{\eta_{2,2,j}} \e}})} \\
&\lesssim \prod_{i=1}^{s} \prod_{j=1}^{l}
\left(1+\norm{\pt^{\eta_{1,1,i}}\inn{\nabla_{v}}^{\norm{\eta_{2,1,i}}+\frac{3}{2}+ \d}g}_{r,L^2}\right) \left(1+\norm{\pt^{\eta_{1,2,j}}\inn{\nabla_{v}}^{\norm{\eta_{2,2,j}}+\frac{3}{2}+ \d}g}_{r,L^2}\right) ,
\end{aligned}
\end{equation*}
for any $\delta>0$ and $r \in \R$. Take $\delta=1/2$. Combining and recalling the norm $\norm{\cdot}_{N,M,r,L^2}$ in \eqref{weighted norm definition}, we thus have 
\begin{equation}\label{bdprod2}
\begin{aligned}
\norm{\prod_{i=1}^{s} \prod_{j=1}^{l} ({\pt^{\eta_{1,1,i}}_{\eta_{2,1,i}} \e})(\overline{{\pt^{\eta_{1,2,j}}_{\eta_{2,2,j}} \e}})}
&\lesssim 
\prod_{(\eta_{1,i,j},\eta_{2,i,j}) \in \G}(1+\norm{g}_{\norm{\eta_{1,i,j}},\norm{\eta_{2,i,j}}+2,r,L^2}).
\end{aligned}
\end{equation}

Next, writing $k = \frac{k\cdot v}{|v|^2}v + k^\perp$ and recalling that $ \d (k \cdot v) $ is the one-dimensional Dirac delta function, we compute 
\begin{align*}
\int_{\R^d} {(k \otimes k) |\widehat{V}(k)|^2 \d (k \cdot v)} \dd k =& \frac{1}{\norm{v}}\int_{\R^d} {(k \otimes k) |\widehat{V}(k)|^2 \d (k \cdot \frac{v}{\norm{v}})}  \dd k \\
=& \frac{1}{\norm{v}}\int_{k^\perp} {(k^\perp \otimes k^\perp) |\widehat{V}(k^\perp)|^2}  \dd k^\perp
= \frac{c_V}{\norm{v}}P^{\perp}_{v},
\end{align*}
for constant $c_V = (2\pi)^{d-2}\int_0^\infty |\hat V(r)| r^{d-2}\; dr$, where $P_v^\perp = \mathbb{I} -\frac{v\otimes v}{|v|^2}.$ Similarly, for any function $\varphi(k)$, we bound 
\begin{align*}
\Big|\int_{\R^d} {(k \otimes k) |\widehat{V}(k)|^2 \d (k \cdot v)} \varphi(k) \dd k\Big| 
&=  \frac{1}{\norm{v}}\Big|\int_{k^\perp} {(k^\perp \otimes k^\perp) |\widehat{V}(k^\perp)|^2} \varphi(k^\perp) \dd k^\perp\Big|
\\
&\le  \frac{c_V}{\norm{v}} \sup_k|\varphi(k)|.  
\end{align*}
Using this, we shall now bound $T^{\Omega_{1},\Omega_{2}}_{\G}$. Indeed, recalling \eqref{TAGpi - definition} and using \eqref{bdprod2}, we bound 
\begin{equation}\label{TAGPi step 0 bound}
\begin{aligned}
\norm{T^{\Omega_{1},\Omega_{2}}_{\G}(g,h)} 
\lesssim & \int_{\R^d} \norm{v-v_{*}}^{-1}  \norm{\pt^{\g_{1}}_{\g_{2}}(\sqrt{\mu_{*}} h_{*}) } \dd v_{*}
\times \prod_{(\eta_{1,i,j},\eta_{2,i,j}) \in \G}(1+\norm{g}_{\norm{\eta_{1,i,j}},\norm{\eta_{2,i,j}}+2,r,L^2}). 
\end{aligned}
\end{equation}
We claim that for any $\delta>0$ and $r\ge 0$, 
\begin{equation}\label{inth}
\int_{\R^d} \norm{v-v_{*}}^{-1}  \norm{\pt^{\g_{1}}_{\g_{2}}(\sqrt{\mu_{*}} h_{*}) } \dd v_{*}
\lesssim
\inn{v}^{-1} \norm{\inn{v}^{-r} \inn{\nabla_{v}}^{\norm{\g_{2}}+\d} \pt^{\g_{1}}h}_{L^2_{v}}.
\end{equation}
Indeed, using the fact that $\mu(v)$ decays rapidly in $v$, we bound 
$$ \int_{\R^d} \norm{v-v_{*}}^{-1}  \norm{\pt^{\g_{1}}_{\g_{2}}(\sqrt{\mu_{*}} h_{*}) } \dd v_{*} \lesssim \Big( \int_{\R^d} \norm{v-v_{*}}^{-p'} \langle v_*\rangle^{-d-1}\Big)^{1/p'}\norm{\inn{v}^{-r} \inn{\nabla_{v}}^{\norm{\g_{2}}} \pt^{\g_{1}}h}_{L^p_{v}}$$
for any $r\ge 0$ and any pair $p, p'$ so that $\frac{1}{p}+\frac{1}{p'}=1$. Now, for any small $\delta>0$, we take $p = 2 + \frac{2\delta}{d-2\delta}$ and $p' = 2 - \frac{2\delta}{d}$, so that the above integral in $v_*$ is finite. Using the Sobolev embedding $H^\delta \subset L^{2 +\frac{2\delta}{d-2\delta}}$, we thus obtain the claim \eqref{inth}, without the pre-factor $\langle v\rangle^{-1}$. To include this pre-factor, it suffices to consider the case when $|v|\ge 1$, and use  
\begin{equation} \label{ineq:v-v*split}
|v-v_*|^{-1} \le 2|v|^{-1} + 2|v|^{-1} \langle v_*\rangle|v-v_*|^{-1}.
\end{equation}
The last inequality is direct by considering $|v-v_*|\ge |v|/2$ and $|v-v_*|\le |v|/2$, and noting that $|v_*|\ge |v|/2$ in the latter case. Therefore, repeating the above calculations, we obtain the claim \eqref{inth} as stated. 

As a result of \eqref{TAGPi step 0 bound} and \eqref{inth}, we have obtained 
\begin{equation} \label{TAGpi step 1 bound}
\begin{aligned}
\norm{T^{\Omega_{1},\Omega_{2}}_{\G}(g,h)} \lesssim \inn{v}^{-1} \norm{\inn{v}^{-r} \pt^{\g_{1}}\inn{\nabla_{v}}^{\norm{\g_{2}}+\d} h}_{L^2_{v}}  \prod_{(\eta_{1,i,j},\eta_{2,i,j}) \in \G}(1+\norm{g}_{\norm{\eta_{1,i,j}},\norm{\eta_{2,i,j}}+2,r,L^2})
,\end{aligned}
\end{equation}
for any $\delta>0$ and $r\ge 0$. Therefore, taking $\delta=1$, for all vector fields $\mathbf{h}_{1}, \mathbf{h}_{2}$, we can 
bound 
\begin{equation}\label{bdhT12}
\begin{aligned}
&\norm{\int_{\R^3}\inn{v}^{-2\th}\mathbf{h}_{1} \cdot T^{\Omega_{1},\Omega_{2}}_{\G}(g,h) \mathbf{h}_{2} \dd v} \\
&\lesssim 
\norm{\inn{v}^{-\frac{1}{2}}\inn{v}^{-\th}\mathbf{h}_{1}}_{L^2_{v}} \norm{\inn{v}^{-\frac{1}{2}} \inn{v}^{-\th}\mathbf{h}_{2}}_{L^2_{v}} \norm{h}_{\norm{\g_{1}},\norm{\g_{2}}+1,r,L^2} \prod_{(\eta_{1,i,j},\eta_{2,i,j}) \in \G}(1+\norm{g}_{\norm{\eta_{1,i,j}},\norm{\eta_{2,i,j}}+2,r,L^2}).
\end{aligned}
\end{equation}
In order to deduce \eqref{T - upper 1}, we need to obtain bounds in terms of the $A$-norm. To this end, we write 
\begin{equation}\label{dechj} \mathbf{h}_{j} = P_v\mathbf{h}_{j} + P_v^\perp\mathbf{h}_{j} \end{equation}
for $j=1,2$, where $P_v, P_v^\perp$ are orthogonal projections onto $v$ and $v^\perp$, respectively, see \eqref{eq:def-projv}. For $P_v^\perp$ components, using \eqref{eq:equiv-L2A}, we bound  
$$
\norm{\inn{v}^{-\frac{1}{2}}\inn{v}^{-\th}P_v^\perp\mathbf{h}_{j}}_{L^2_{v}} \le \norm{\mathbf{h}_{j}}_{0,0,\th,A}.$$
For $P_v$ components, we note that  
due to the presence of the Dirac delta $\d (k \cdot (v-v_{*}))$ in the collision kernel, we may write  
\begin{equation} \label{equation:v v* relation}
\begin{split}
(k \otimes k) v \d (k \cdot (v-v_{*})) &= (k \otimes k) v_{*} \d (k \cdot (v-v_{*})), 
\\
v\cdot(k \otimes k)  \d (k \cdot (v-v_{*})) &= v_*\cdot (k \otimes k) \d (k \cdot (v-v_{*})).
\end{split}
\end{equation}
Namely, the very same bound \eqref{TAGpi step 1 bound} remains to hold when $T^{\Omega_{1},\Omega_{2}}_{\G}$ is replaced by $T^{\Omega_{1},\Omega_{2}}_{\G}v$ or $v\cdot T^{\Omega_{1},\Omega_{2}}_{\G}$, upon using the rapid decay of $\mu(v_*)$ in $v_*$. As a result, we gain an extra factor of $\langle v\rangle^{-1}$ for the terms involving $P_v$, which are sufficient to deduce the bounds in the $A$-norm, thanks to \eqref{eq:equiv-L2A}. This yields  
\begin{equation*}
\begin{aligned}
&\norm{\int_{\R^3}\inn{v}^{-2\th}\mathbf{h}_{1} \cdot T^{\Omega_{1},\Omega_{2}}_{\G}(g,h) \mathbf{h}_{2} \dd v} \\
&\lesssim 
\norm{\mathbf{h}_{1}}_{0,0,\th,A} \norm{\mathbf{h}_{2}}_{0,0,\th,A} \norm{h}_{\norm{\a_{2}},\norm{\b_{2}}+1,r,L^2} \prod_{(\eta_{1,i,j},\eta_{2,i,j}) \in \G}(1+\norm{g}_{\norm{\eta_{1,i,j}},\norm{\eta_{2,i,j}}+2,r,L^2}),
\end{aligned}
\end{equation*}
giving \eqref{T - upper 1} as stated.

Next, we prove \eqref{T - upper 2}; namely, we may exchange one derivative of $h$ with one derivative of $\mathbf{h}_{2}$ in the previous estimates. Precisely, recalling \eqref{TAGPi step 0 bound}, in the case $d \ge 3$, the term $|v - v_*|^{-1}$ belongs to $L^2_{loc}$. Thus, we do not need to lose any derivatives in \eqref{TAGpi step 1 bound}, and we obtain the following result, consistent with the previous step:
\begin{equation*}
\begin{aligned}
&\norm{\int_{\R^3}\inn{v}^{-2\th}\mathbf{h}_{1} \cdot T^{\Omega_{1},\Omega_{2}}_{\G}(g,h) \mathbf{h}_{2} \dd v} \\
&\lesssim  \norm{\mathbf{h}_{1}}_{0,0,\th,A} \norm{\mathbf{h}_{2}}_{0,0,\th,A} \norm{h}_{\norm{\a_{2}},\norm{\b_{2}},r,L^2} \prod_{(\eta_{1,i,j},\eta_{2,i,j}) \in \G}(1+\norm{g}_{\norm{\eta_{1,i,j}},\norm{\eta_{2,i,j}}+2,r,L^2}).
\end{aligned}
\end{equation*}
which in particular yields \eqref{T - upper 2}, since $\norm{\mathbf{h}_{2}}_{0,0,\th,A}  \le \norm{\mathbf{h}_{2}}_{0,1,\th,A} $. On the other hand, in the case $d=2$, the mapping $v_{*} \mapsto |v-v_{*}|^{-1}$ does not belong to $L^2_{\text{loc}}(\mathbb{R}^d)$. Therefore, by the Hardy-Littlewood-Sobolev inequality, we obtain the following estimate:
\begin{align*}
\norm{\int_{\R^2} \norm{\cdot-v_{*}}^{-1}  \norm{\pt^{\a_{2}}_{\b_{2}}(\sqrt{\mu_{*}} h_{*}) } \dd v_{*}}_{L^4} \lesssim \norm{\pt^{\a_{2}}_{\b_{2}}(\sqrt{\mu_{*}} h_{*}) }_{L^{4/3}} \lesssim  \sum_{\g \le \b_{2}}\norm{\langle v_*\rangle^{-r}\pt^{\a_{2}}_{\g} h_{*} }_{L^2},
\end{align*}
upon using the rapid decay of $\mu(v)$ in $v$. 
Therefore, recalling \eqref{TAGPi step 0 bound}, we obtain 
$$
\begin{aligned}
\norm{T^{\Omega_{1},\Omega_{2}}_{\G}(g,h)}_{L^4_v} 
\lesssim & \norm{h}_{\norm{\a_{2}},\norm{\b_{2}},r,L^2} \prod_{(\eta_{1,i,j},\eta_{2,i,j}) \in \G}(1+\norm{g}_{\norm{\eta_{1,i,j}},\norm{\eta_{2,i,j}}+2,r,L^2}). 
\end{aligned}
$$
Next, following a similar argument as done for \eqref{bdhT12}, we obtain 
\begin{equation*}
\begin{aligned}
&\norm{\int_{\R^2}\inn{v}^{-2\th}\mathbf{h}_{1} \cdot T^{\Omega_{1},\Omega_{2}}_{\G}(g,h) \mathbf{h}_{2} \dd v}
\\&\lesssim 
\norm{\inn{v}^{-\th} \inn{v}^{-\frac{1}{2}}\mathbf{h}_{1}}_{L^2_v}
\norm{\inn{v}^{-\th} \inn{v}^{-\frac{1}{2}}\mathbf{h}_{2}}_{L^4_v} \norm{T^{\Omega_{1},\Omega_{2}}_{\G}(g,h)}_{L^4_v}\\
&\lesssim
\norm{\inn{v}^{-\th} \inn{v}^{-\frac{1}{2}}\mathbf{h}_{1}}_{L^2_v}
\norm{\inn{v}^{-\th} \inn{v}^{-\frac{1}{2}}\mathbf{h}_{2}}_{H^1_v}  \norm{h}_{\norm{\a_{2}},\norm{\b_{2}},r,L^2} \prod_{(\eta_{1,i,j},\eta_{2,i,j}) \in \G}(1+\norm{g}_{\norm{\eta_{1,i,j}},\norm{\eta_{2,i,j}}+2,r,L^2}),
\end{aligned}
\end{equation*}
in which we have used the Sobolev embedding $\norm{\inn{v}^{-\th} \inn{v}^{-\frac{1}{2}}\mathbf{h}_{2}}_{L^4_v} 
\lesssim \norm{\inn{v}^{-\th} \inn{v}^{-\frac{1}{2}}\mathbf{h}_{2}}_{H^1_v}$ which is valid in $\R^2$ with weights $\langle v\rangle^a$ since $\partial_v \langle v\rangle^a \lesssim \langle v\rangle^a$ for any $a\in \R$. Finally, to deduce the bounds in terms of the $A$-norm, we proceed as in the previous case, namely decomposing $\mathbf{h}_j$ as in \eqref{dechj} and using \eqref{equation:v v* relation}. This yields \eqref{T - upper 2}.

Finally, we focus on \eqref{T - upper 3}. Note that this is the case when all the derivative act on the $\e$ with $\eta_{1,1,1} = \a, ~ \eta_{2,1,1} = \b$ or $\eta_{1,2,1} = \a,~  \eta_{2,2,1} = \b$. Precisely, we have  
\begin{equation*}
T^{\Omega_{1},\Omega_{2}}_{\G}(g,h) =
\begin{cases}
& -\iint_{\R^d \times \R^d} (k \otimes k)  \widehat{V}(k)^2 \d (k \cdot (v-v_{*})) \frac{1}{\norm{\e}^{4}} 
({\pt^{\a}_{\b} \e})(\overline{ \e})
(\sqrt{\mu_{*}}h_{*}) \dd v_{*} \dd k,
\\
&
-\iint_{\R^d \times \R^d} (k \otimes k) \widehat{V}(k)^2 \d (k \cdot (v-v_{*})) \frac{1}{\norm{\e}^{4}} 
({\e})(\overline{ \pt^{\a}_{\b} \e})
(\sqrt{\mu_{*}}h_{*}) \dd v_{*} \dd k.
\end{cases}
\end{equation*}
The two cases are similar, and we will only consider the first case. By the definition of $\e$, see \eqref{def-efunction}, we have
\begin{equation*}
\pt^{\a}_{\b} \e = \widehat{V}(k)\int_{\R^d} \frac{\hat{k}\cdot \nabla_{v}\pt_{\b}^{\a}(\sqrt{\mu} g)_{**}}{\hat{k} \cdot (v-v_{**})- i0^+} \dd v_{**},
\end{equation*}
with $\hat{k} = k/\norm{k}$. 
Thus, we write 
\begin{equation}\label{writeTh12}
\begin{aligned}
\int_{\R^3}\inn{v}^{-2\th}\mathbf{h}_{1} \cdot T^{\Omega_{1},\Omega_{2}}_{\G}(g,h)\mathbf{h}_{2} \dd v = \int_{\R^d} (k \otimes k)  \widehat{V}(k)^2 \Big( \iint_{\R^d\times \R^d}\d (k \cdot (v-v_{*})) \mathcal{E}(k,v,v_*)
\; \dd v \dd v_* \Big) \; \dd k
\end{aligned}
\end{equation} 
in which $\mathcal{E}(k,v,v_*)$ satisfies 
\begin{equation}\label{bdEc}|\mathcal{E}(k,v,v_*)
| \lesssim \norm{\inn{v}^{-\th}\mathbf{h}_{1}} \norm{\inn{v}^{-\th}\mathbf{h}_{2}} \sqrt{\mu_{*}} \norm{h_{*}}  \norm{\int_{\R^d} \frac{\hat{k}\cdot \nabla_{v}\pt_{\b}^{\a}(\sqrt{\mu} g)_{**}}{\hat{k} \cdot (v-v_{**})- i0^+} \dd v_{**}}. 
\end{equation}
Fix $k\in \R^d \setminus\{0\}$, and set $\hat{k} = k/\norm{k}$. We first study the integrals over $v,v_*\in\R^d$ in the expression \eqref{writeTh12}. Indeed, we may write $\R^d = \hat k\R\oplus\hat k^\bot$ with the decomposition $v= u\hat{k}+ \eta$ and $v_{*}= u_{*}\hat{k}+ \eta_{*}$ for $u,u_*\in \R$ and $\eta, \eta_* \in k^\bot$. Note that the maps $v\mapsto (u,\eta)$ and $v_*\mapsto (u_*,\eta)$ have Jacobian determinant equal to one. Therefore, recalling that $\d (\cdot)$ is the one-dimensional Dirac delta function, we compute 
$$
\begin{aligned} 
\iint_{\R^d\times \R^d}\d (k \cdot (v-v_{*})) \mathcal{E}(k,v,v_*)
\; \dd v \dd v_* 
&= \iint_{(\hat k\R\oplus\hat k^\bot)^2}\frac{1}{|k|}\d (u-u_*) \mathcal{E}(k,u\hat{k}+ \eta,u_*\hat{k}+ \eta_*)
\;  \dd u \dd\eta \dd u_{*} \dd \eta_{*}
\\
&= \frac{1}{|k|}\iint_{\R\times k^\perp\times k^\perp}\mathcal{E}(k,u\hat{k}+ \eta, u\hat{k}+ \eta_*)
\;  \dd u \dd\eta \dd \eta_{*}.
\end{aligned}$$
Therefore, for all $r_{0}\ge0$, using \eqref{bdEc} and the H\"older's inequality, we bound 
\begin{equation*}
\begin{aligned}
&\Big| \iint_{\R^d\times \R^d}\d (k \cdot (v-v_{*})) \mathcal{E}(k,v,v_*)
\; \dd v \dd v_* \Big|
\\
&\le \frac{1}{\norm{k}} \iint_{\R\times k^\perp\times k^\perp} 
\norm{\inn{u\hat{k}+\eta}^{-\th}\mathbf{h}_{1}(u,\eta)} \norm{\inn{u\hat{k}+\eta_{}}^{-\th}\mathbf{h}_{2}(u,\eta_{})} \\
&\qquad \times
(\sqrt{\mu} \norm{h})(u,\eta_{*}) \norm{\int_{\R^d} \frac{\hat{k}\cdot \nabla_{v}\pt_{\b}^{\a}(\sqrt{\mu} g)_{**}}{u-\hat{k} \cdot v_{**}- i0^+} \dd v_{**}} \dd u \dd\eta \dd \eta_{*} \\
&\lesssim \frac{1}{\norm{k}} \norm{\inn{\hat{k}\cdot v}^{-r_{0}} \inn{v}^{-\th}\mathbf{h}_{1}}_{L^2(\hat{k}^{\perp};L^2(\hat{k}\R))} \norm{\inn{\hat{k}\cdot v}^{-r_{0}} \inn{v}^{-\th}\mathbf{h}_{2}}_{L^2(\hat{k}^{\perp};L^\infty(\hat{k}\R))} 
\\
&\qquad \times \norm{\inn{\hat{k}\cdot v}^{2 r_{0}} \sqrt{\mu}h}_{L^1(\hat{k}^{\perp};L^\infty(\hat{k}\R))} \norm{\int_{\R^d} \frac{\hat{k}\cdot \nabla_{v}\pt_{\b}^{\a}(\sqrt{\mu} g)_{**}}{u - \hat{k} \cdot v_{**}- i0^+} \dd v_{**}}_{L^2_u(\hat{k}\R)},
\end{aligned}
\end{equation*}
in which $v= u\hat{k}+ \eta$ and $v_{*}= u_{*}\hat{k}+ \eta_{*}$. Now, using the boundedness of the Hilbert transform in $L^2(\hat k\R)$ and the rapid decay of $\mu$, we deduce for all $r\ge0$,
$$ \norm{\int_{\R^d} \frac{\hat{k}\cdot \nabla_{v}\pt_{\b}^{\a}(\sqrt{\mu} g)_{**}}{u - \hat{k} \cdot v_{**}- i0^+} \dd v_{**}}_{L^2_u(\hat{k}\R)} \lesssim \norm{\inn{v}^{-r} \pt^{\a} \inn{\nabla_{v}}^{\norm{\b}+1}g}_{L^2}.
$$
This, together with the Sobolev embedding $H^{1/2+\delta} \subset L^\infty$ on $\hat k\R$, for any $\delta>0$ and the fact that $|\hat{k}\cdot v| \le |v|$, yields
\begin{equation*}
\begin{aligned}
&\Big| \iint_{\R^d\times \R^d}\d (k \cdot (v-v_{*})) \mathcal{E}(k,v,v_*)
\; \dd v \dd v_* \Big|
\\&\lesssim 
\frac{1}{\norm{k}} \norm{\inn{\hat{k}\cdot v}^{-r_{0}} \inn{v}^{-\th}\mathbf{h}_{1}}_{L^2_{v}} \norm{\inn{\hat{k}\cdot v}^{-r_{0}} \inn{v}^{-\th}\inn{\nabla_{v}}^{\frac{1}{2}+\d}\mathbf{h}_{2}}_{L^2_{v}} 
\\
&\quad \times \norm{\inn{v}^{-r} \inn{\nabla_{v}}^{\frac{1}{2} +\d}h}_{L^2} \norm{\inn{v}^{-r} \pt^{\a} \inn{\nabla_{v}}^{\norm{\b}+1}g}_{L^2}.
\end{aligned}
\end{equation*}
Using this, we are now ready to bound \eqref{writeTh12}. Indeed, choosing $r_0 >\frac{1}{2}$, we note that
\begin{equation*}
\int_{\mathbb{S}^{d-1}} \inn{\hat{k} \cdot v}^{-2r_{0}} \dd \sigma(\hat{k}) = \norm{\mathbb{S}^{d-2}} \int_{0}^{\pi} (1+ \norm{v}^2 \cos^2 \th)^{-r_0} \sin \th \dd \th \lesssim \inn{v}^{-1},
\end{equation*}
and therefore, recalling that $V$ is symmetric and $V \in\dot{H}^{\frac{1}{2}}$, we obtain 
\begin{equation*}
\begin{aligned}
&\iint_{\R \times \R^d}\norm{k} \widehat{V}(k)^2 \norm{\inn{\hat{k}\cdot v}^{-r_{0}} \inn{v}^{-\th}\mathbf{h}_{1}}_{L^2_{v}} \norm{\inn{\hat{k}\cdot v}^{-r_{0}} \inn{v}^{-\th}\inn{\nabla_{v}}^{\frac{1}{2}+\d}\mathbf{h}_{2}}_{L^2_{v}} \dd k\\
&\lesssim  \left(\int_{\R^d}\norm{k} \widehat{V}(k)^2 \inn{\hat{k}\cdot v}^{-2r_{0}} \inn{v}^{-2\th}\norm{\mathbf{h}_{1}}^2 \dd v \dd k \right)^{\frac{1}{2}}\\
&\qquad \times \left(\int_{\R^d}\norm{k} \widehat{V}(k)^2 \inn{\hat{k}\cdot v}^{-2r_{0}} \inn{v}^{-2\th}\norm{\inn{\nabla_{v}}^{\frac{1}{2}+\d} \mathbf{h}_{2}}^2 \dd v \dd k\right)^{\frac{1}{2}}\\
&\lesssim \norm{\inn{v}^{-\frac{1}{2}}\inn{v}^{-\th}\mathbf{h}_{1}}_{L^2_{v}} \norm{\inn{v}^{-\frac{1}{2}}\inn{v}^{-\th} \inn{\nabla_{v}}^{\frac{1}{2}+\d}\mathbf{h}_{2}}_{L^2_{v}} .
\end{aligned}
\end{equation*}
This proves that 
$$
\begin{aligned}
\Big|\int_{\R^3}\inn{v}^{-2\th}\mathbf{h}_{1} \cdot T^{\Omega_{1},\Omega_{2}}_{\G}(g,h)\mathbf{h}_{2} \dd v
\Big| 
&\lesssim \norm{\inn{v}^{-\frac{1}{2}}\inn{v}^{-\th}\mathbf{h}_{1}}_{L^2_{v}} \norm{\inn{v}^{-\frac{1}{2}}\inn{v}^{-\th} \inn{\nabla_{v}}^{\frac{1}{2}+\d}\mathbf{h}_{2}}_{L^2_{v}} 
\\&\qquad  \times \norm{\inn{v}^{-r} \inn{\nabla_{v}}^{\frac{1}{2} +\d}h}_{L^2} \norm{\inn{v}^{-r} \pt^{\a} \inn{\nabla_{v}}^{\norm{\b}+1}g}_{L^2}
\end{aligned}$$
for any $r\ge 0$, $\delta>0$, and $\theta \in \R$, which proves \eqref{T - upper 3} if $\mathbf{h}_{j} = P_v^\perp \mathbf{h}_{j}$ (upon taking $\delta=1/2$). In general, we use the decomposition \eqref{dechj} and the relation \eqref{equation:v v* relation} to gain an extra factor of $\langle v\rangle^{-1}$ for $P_v$ components as done in the previous case. This completes the proof of the lemma. 
\end{proof}

To simplify the notation, we will denote $\e(k,k\cdot v ; \nabla_{v} F_{g})$ and $\e(k,k\cdot v ; \nabla_{v} \mu)$ by $\e_{g}$ and $\e_{\mu}$, respectively. Similar to $\mathcal{B}_{\nabla_{v} F_g}[h]$, we can get the estimate of $\mathcal{B}_{\nabla_{v} F_g}[h] - \mathcal{B}_{\nabla_{v} \mu}[h]$.
\begin{equation} \label{B-B leibniz rule}
\pt_{\b}^{\a} (\mathcal{B}_{\nabla_{v} F_g}[h] - \mathcal{B}_{\nabla_{v} \mu}[h]) = \sum_{\substack{\a_{1} + \a_{2} = \a\\ \b_{1} + \b_{2} = \b}} \binom{\a}{\a_{1}} \binom{\b}{\b_{1}} H^{\a_1,\b_1,\a_{2},\b_{2}} (g,h)
\end{equation}
where 
\begin{equation}\label{defGBOG}
\begin{aligned}
H^{\a_1,\b_1,\a_{2},\b_{2}} (g,h) :&= \iint_{\R^d \times \R^d} (k \otimes k) |\widehat{V}(k)|^2 \d (k \cdot (v-v_{*}))  \\
&\qquad\times \pt_{\b_{1}}^{\a_{1}} \left( \frac{(\e_{g}-\e_{\mu})\overline{\e_{g}} + (\overline{\e_{g}}-\overline{\e_{\mu}}){\e_{\mu}}}{\norm{\e_{g}}^2 \norm{\e_{\mu}}^2}\right)  \pt_{\a_{2}}^{\b_{2}} (\sqrt{\mu_{*}}h_{*}) \dd v_{*} \dd k.
\end{aligned}
\end{equation}
Here, we compute 
\begin{align*}
&\pt_{\b_{1}}^{\a_{1}} \left( \frac{(\e_{g}-\e_{\mu})\overline{\e_{g}} + (\overline{\e_{g}}-\overline{\e_{\mu}}){\e_{\mu}}}{\norm{\e_{g}}^2 \norm{\e_{\mu}}^2}\right)\\ 
\\&= \sum_{\substack{ \eta_{1,1} + \eta_{1,2} = \a_{1}\\ \ \eta_{2,1} +  \eta_{2,2} = \b_{1}}}{C_{\eta}}
\Bigg(\Big( {\pt^{\eta_{1,1}}_{\eta_{2,1}} (\e_{g}-\e_{\mu})}{{\pt^{\eta_{1,2}}_{\eta_{2,2}} \frac{\overline{\e_{g}}}{\norm{\e_{g}}^2 \norm{\e_{\mu}}^2}}}\Big)
+ \Big( {\pt^{\eta_{1,1}}_{\eta_{2,1}} (\overline{\e_{g}}-\overline{\e_{\mu}})}{{\pt^{\eta_{1,2}}_{\eta_{2,2}} \frac{{\e_{\mu}}}{\norm{\e_{g}}^2 \norm{\e_{\mu}}^2}}}\Big) \Bigg)
\end{align*}
Similar to \eqref{expGOmega} and \eqref{TAGpi - definition}, we can define
\begin{equation}\label{expHOmega}
H^{\Omega_{1},\Omega_{2}} (g,h) = \sum_{\Gamma} C_{\Omega_{1},\Omega_{2},\Gamma} S^{\Omega_{1},\Omega_{2}}_{\G}(g,h),
\end{equation}
where 
\begin{equation} \label{SAGpi - definition}
\begin{aligned}
S^{\Omega_{1},\Omega_{2}}_{\G}(g,h) :&= \iint_{\R^d \times \R^d} (k \otimes k) |\widehat{V}(k)|^2 \d (k \cdot (v-v_{*})) (\pt^{\eta_{1,1}}_{\eta_{2,1}} (\e_{g}-\e_{\mu})+ \pt^{\eta_{1,1}}_{\eta_{2,1}} (\overline{\e_{g}}-\overline{\e_{\mu}}))\\
&\qquad\times  \frac{C_{\eta}}{{\e_{g}}^{s+1} \bar{\e_{g}}^{l+1} {\e_{\mu}}^{p+1}\bar{\e_{\mu}}^{q+1}}\\
&\qquad\times 
\prod ({\pt^{\eta_{1,1,i}}_{\eta_{2,1,i}} \e_{g}})(\overline{{\pt^{\eta_{1,2,j}}_{\eta_{2,2,j}} \e_{g}}}) ({\pt^{\eta_{1,3,k}}_{\eta_{2,3,k}} \e_{\mu}})(\overline{{\pt^{\eta_{1,4,t}}_{\eta_{2,4,t}} \e_{\mu}}})
\pt_{\a_{2}}^{\b_{2}} (\sqrt{\mu_{*}}h_{*}) \dd v_{*} \dd k.
\end{aligned}
\end{equation}
Similar to Lemma \ref{lemma:T - upper bound}, we have the following lemma. Since the proof follows the same argument as that of Lemma \ref{lemma:T - upper bound}, we thus avoid repeating the details.

\begin{lemma} \label{lemma:S - upper bound}
(Upper bound of $S^{\Omega_{1},\Omega_{2}}_{\G}$) 
Assume that the same conditions as in Lemma \ref{lemma:T - upper bound} hold for $V$ and $g$. Then, we have for all vector fields $\mathbf{h}_{1}$, $\mathbf{h}_{2}$, for all multi-index $\Omega_{1}$, $\Omega_{2}$, $\G$ and real number $r \ge 0$, and $\th \in \R$
\begin{equation}\label{S - upper 1}
\begin{aligned}
&\norm{\int_{\R^d} \inn{v}^{-2\th} \mathbf{h}_{1} \cdot S^{\Omega_{1},\Omega_{2}}_{\G}(g,h) \mathbf{h}_{2} \dd v} \\
&\lesssim 
\norm{\mathbf{h}_{1}}_{0,0,\th,A} \norm{\mathbf{h}_{2}}_{0,0,\th,A} \norm{h}_{\norm{\a_{2}},\norm{\b_{2}}+1,r,L^2} \norm{g}_{\norm{\eta_{1,1}},\norm{\eta_{2,1}}+2,r,L^2} \prod_{(\eta_{1,i,j},\eta_{2,i,j}) \in \G}(1+\norm{g}_{\norm{\eta_{1,i,j}},\norm{\eta_{2,i,j}}+2,r,L^2}).
\end{aligned}
\end{equation}
Alternatively, we can exchange one derivative of $h$ with one derivative of $\mathbf{h}_{2}$, namely
\begin{equation}\label{S - upper 2}
\begin{aligned}
&\norm{\int_{\R^d}\inn{v}^{-2\th} \mathbf{h}_{1} \cdot S^{\Omega_{1},\Omega_{2}}_{\G}(g,h) \mathbf{h}_{2} \dd v} \\
&\lesssim 
\norm{\mathbf{h}_{1}}_{0,0,\th,A} \norm{\mathbf{h}_{2}}_{0,1,\th,A} \norm{h}_{\norm{\a_{2}},\norm{\b_{2}},r,L^2} \norm{g}_{\norm{\eta_{1,1}},\norm{\eta_{2,1}}+2,r,L^2}\prod_{(\eta_{1,i,j},\eta_{2,i,j}) \in \G}(1+\norm{g}_{\norm{\eta_{1,i,j}},\norm{\eta_{2,i,j}}+2,r,L^2}).
\end{aligned}
\end{equation}
In addition, if all the derivatives (with respect to $\a$ and $\b$) act on exactly one $\e$, we can also exchange one derivative of $h$ with one derivative of $g$, namely
\begin{equation}\label{S - upper 3}
\begin{aligned}
\norm{\int_{\R^d}\inn{v}^{-2\th} \mathbf{h}_{1} \cdot S^{\Omega_{1},\Omega_{2}}_{\G}(g,h) \mathbf{h}_{2} \dd v} \lesssim 
\norm{\mathbf{h}_{1}}_{0,0,\th,A} \norm{\mathbf{h}_{2}}_{0,1,\th,A} 
\norm{h}_{0,1,r,L^2} \norm{g}_{\norm{\a},\norm{\b}+1,r,L^2}.
\end{aligned}
\end{equation}
\end{lemma}

\subsection{Nonlinear Estimates}

In this section, we establish nonlinear estimates in weighted spaces. Precisely, we obtain the following lemma. 

\begin{lemma} \label{lem-boundN} Let $N(g_1,g_2,g_3)$ be the nonlinear term defined as in \eqref{defNg123}. Assume that $V$ and $g_1, g_2, g_3$ satisfy the same conditions as in Proposition \ref{proposition:N upper bound}. Then, for any $r \ge 0$, $\theta \in \R$, and any multi-index $\a$ and $\b$, there hold
\begin{equation} \label{cor 7 - inequality 1}
\begin{aligned}
\norm{\int_{\R^d} \inn{v}^{-2\th} h\pt^{\a}_{\b}N(g_{1},g_{2},g_{3}) \dd v } 
&\lesssim \norm{h}_{0,0,\th,D} \sum \Big(\norm{g_{3}}_{\norm{\a_{3}},\norm{\g_{3}},\th,D} \norm{g_{2}}_{\norm{\a_{2}},\norm{\g_{2}},r,D} \prod_{j}(1+\norm{g_{1}}_{\norm{\a_{1,j}},\norm{\g_{1,j}},r,D}) \\
&\quad+\norm{g_{3}}_{\norm{\a_{3}},\norm{\g_{3}},\th,D} \norm{g_{1}}_{\norm{\a_{1,1}},\norm{\g_{1,1}},r,D}  \prod_{j}(1+\norm{g_{1}}_{\norm{\a_{1,j}},\norm{\g_{1,j}},r,D}) \\
&\quad+ \norm{g_{2}}_{\norm{\a_{2}},\norm{\g_{2}},r,D} \norm{g_{1}}_{\norm{\a_{1,1}},\norm{\g_{1,1}},r,D} \prod_{j}(1+\norm{g_{1}}_{\norm{\a_{1,j}},\norm{\g_{1,j}},r,D})\Big),
\end{aligned}
\end{equation}
where the summation is taken over all partitions of the multi-indices $\a$ and $\b$ satisfying $\sum_{j}\a_{1,j} + {\a_{2}} + {\a_{3}} = {\a}$ and $ \sum_{j}\b_{1,j} + {\b_{2}} + {\b_{3}} = {\b},
$
and the sequence $\{\gamma_{i}\}$ is defined as follows: 

\begin{itemize}

\item For any $i \in \{(1,1),\cdots (1,n),2\}$, set $\gamma_i = |\beta_i|+1$ if $(\a_i,\b_i)\not = (\a,\b)$, otherwise set $\gamma_i = |\b|$. 

\item For $i=3$, set $\gamma_3 = \beta_3$ if $(\a_3,\b_3)\not =0$, otherwise set $\gamma_3=1$. 
 
\end{itemize}

\end{lemma}

\begin{remark} 
Note that the estimates stated in \eqref{cor 7 - inequality 1} are pointwise in $x$, and therefore no loss of spatial derivatives is present. In addition, the loss of one $v$-derivative is allowed, except for top derivatives (i.e. when $(\alpha_i,\beta_i) = (\alpha,\beta)$) and for $i=3$ (i.e. when derivatives hit $g_3$).  

\end{remark}

\begin{proof}[Proof of Lemma \ref{lem-boundN}] In view of \eqref{defNg123}, we write 
$$N = \sum_{j=1}^4 N_{j},$$
where 
$$
\begin{aligned} 
N_1 &= (\nabla_{v} - v) \cdot (\mathcal{B}(\nabla_{v} F_{g_{1}})[g_{2}] \nabla_{v} g_{3})
,\\
N_2 &= - (\nabla_{v} -v) \cdot ( \mathcal{B}(\nabla_{v} F_{g_{1}}) [\nabla_{v} g_{2}]g_{3}),
\\
N_3&=(\nabla_{v} -v) \cdot \Big ( \mathcal{B}_{\nabla_{v} F_{g_{1}}}[\sqrt{\mu}] - \mathcal{B}_{\nabla_{v} \mu} [\sqrt{\mu}] \Big) (\nabla_{v} + v)g_{3},
\\N_4 &= 
- (\nabla_{v} -v) \cdot \Bigg(\sqrt{\mu} \Big( \mathcal{B}_{\nabla_{v} F_{g_{1}}} [(\nabla_{v} +v) g_{2}] - \mathcal{B}_{\nabla_{v} \mu} [(\nabla_{v} +v) g_{2}] \Big)\Bigg).
\end{aligned}$$
Let us first treat the nonlinear term $N_1$. By definition, we compute 
\begin{align}
\int_{\R^d}
 \inn{v}^{-2\th} h\pt^{\a}_{\b}N_1 \dd v  
=& \int_{\R^d}\inn{v}^{-2\th} h \pt^{\a}_{\b} \left((\nabla_{v} - v) \cdot \mathcal{B}(\nabla_{v} F_{g_{1}})[g_{2}] \nabla_{v} g_{3} \right)\dd v  \nonumber \\
=& -\int_{\R^d}\inn{v}^{-2\th} (\nabla_{v} + v)h \cdot \pt^{\a}_{\b}\left(\mathcal{B}(\nabla_{v} F_{g_{1}})[g_{2}] \nabla_{v} g_{3}\right) \dd v  \label{cor 7 pf eq 0}\\
&- \int_{\R^d}\sum_{e_{i} \le \b} \inn{v}^{-2\th} h e_{i}  \pt^{\a}_{\b-e_{i}}\left(\mathcal{B}(\nabla_{v} F_{g_{1}})[g_{2}] \nabla_{v} g_{3}\right)_{i} \dd v  \label{cor 7 pf eq 1}\\
&- \int_{\R^d}(\nabla_{v} \inn{v}^{-2\th}) h \cdot \pt^{\a}_{\b}\left(\mathcal{B}(\nabla_{v} F_{g_{1}})[g_{2}] \nabla_{v} g_{3}\right) \dd v .\label{cor 7 pf eq 2}
\end{align}
Recalling \eqref{B leibniz rule}, we compute
\begin{equation} \label{corollary 8 -eq1}
\begin{aligned}
&\int_{\R^d} \inn{v}^{-2\th} (\nabla_{v} +v) h \cdot \pt^{\a}_{\b} \big(\mathcal{B}(\nabla_{v} F_{g_{1}})[g_{2}] \nabla_{v} g_{3} \big) \; \dd v  \\
&= \sum_{\substack{\a_{1} +\a_{2} +\a_{3} = \a\\ \b_{1} +\b_{2} +\b_{3} = \b} }  C_{\a_{i},\b_{i}}\int_{\R^d}\inn{v}^{-2\th} (\nabla_{v} + v)h \cdot G^{\a_{1},\b_{1},\a_2, \b_2}(g_{1},g_{2}) \nabla_{v}\pt^{\a_{3}}_{\b_{3}}g_{3} \; \dd v,
\end{aligned}
\end{equation}
for $G^{\Omega_{1},\Omega_{2}}$ being defined as in \eqref{defGOG}. Consider the case when $\alpha_1 = \alpha$ and $\beta_1 = \beta$ (and hence $\alpha_2=\alpha_3=0$ and $\beta_2=\beta_3=0$).  This case can be decomposed as in \eqref{expGOmega}. If all derivatives are applied to either $\epsilon$ or $\bar{\epsilon}$ (i.e., when $\eta_{1,1,1} = \alpha, \eta_{2,1,1} = \beta$ or $\eta_{1,2,1} = \alpha, \eta_{2,2,1} = \beta$), we apply \eqref{T - upper 3}, yielding the desired estimates. Here, recalling  \eqref{equivD}, we note that 
\begin{equation}\label{vdnorm}\norm{g_1}_{\norm{\a},\norm{\b}+1,r + \frac32,L^2} \lesssim  \norm{g_1}_{\norm{\a},\norm{\b},r,D} \end{equation}
which has no loss of $v$-derivatives. If otherwise (i.e. not all derivatives hit $\epsilon$ or $\bar{\epsilon}$), we use \eqref{T - upper 2} to obtain the following estimate, for any $r \ge 0$. 
\begin{equation} \label{corollary 8 -eq2}
\begin{aligned}
&\Big| \int_{\R^d}\inn{v}^{-2\th} (\nabla_{v} + v)h \cdot G^{\a,\b,0,0} (g_{1},g_{2}) \nabla_{v}g_{3} \; \dd v\Big|
\\
&\lesssim \norm{(\nabla_{v} +v)h}_{0,0,\th,A} \norm{\nabla_{v} g_{3}}_{0,1,\th,A} 
\norm{g_{2}}_{0,1,r+\frac32,L^2}\\
& \quad \times \Big(\norm{g_{1}}_{\norm{\a},\norm{\b}+1,r+\frac32,L^2} + \prod_{\substack{\sum M_{j} = \norm{\a} \\
\sum N_{j} = \norm{\b} \\N_{j}+M_{j} < \norm{\a} + \norm{\b}}}(1+\norm{g_1}_{N_{j},M_{j}+2,r+\frac32,L^2})\Big),
\end{aligned}
\end{equation}
noting the constraint $N_{j}+M_{j} < \norm{\a} + \norm{\b}$ holds due to the fact not all derivatives hit on $\epsilon$ or $\bar{\epsilon}$. Using again \eqref{vdnorm}, we obtain the lemma for this case. 

Next, we consider the case when $\alpha_2 = \alpha$ and $\beta_2 = \beta$, using \eqref{T - upper 2}, for any $r \ge 0$, we bound 
\begin{equation} \label{corollary 8 -eq3}
\begin{aligned}
&\int_{\R^d}\inn{v}^{-2\th} (\nabla_{v} + v)h \cdot G^{0,0,\a,\b} (g_{1},g_{2}) \nabla_{v}g_{3} \; \dd v
\\&\lesssim \norm{(\nabla_{v} +v)h}_{0,0,\th,A} \norm{\nabla_{v} g_{3}}_{0,1,\th,A}\norm{g_{2}}_{\norm{\a},\norm{\b},r+\frac32,L^2} (1+\norm{g_1}_{0,2,r+\frac32,L^2}),
\end{aligned}
\end{equation}
which again yields the lemma in this case. Finally, in all the remaining cases, using \eqref{T - upper 1}, for any $r \ge 0$, we bound 
\begin{equation} \label{corollary 8 -eq4}
\begin{aligned}
&\int_{\R^d}\inn{v}^{-2\th} (\nabla_{v} + v)h \cdot G^{\a_{1},\b_{1},\a_{2},\b_{2}} (g_{1},g_{2}) \nabla_{v}\pt^{\a_{3}}_{\b_{3}}g_{3} \; \dd v
\\&\lesssim \norm{(\nabla_{v} +v)h}_{0,0,\th,A} \norm{\nabla_{v} g_{3}}_{\norm{\a_{3}},\norm{\b_{3}},\th,A} \norm{g_{2}}_{\norm{\a_{2}},\norm{\b_{2}}+1,r+\frac32,L^2}\prod_{\substack{\sum M_{j} = \norm{\a_{1}} \\
\sum N_{j} = \norm{\b_{1}}}}(1+\norm{g_{1}}_{N_{j},M_{j}+2,r+\frac32,L^2}).
\end{aligned}
\end{equation}

We note that the above estimates have no impact on the $x$-derivatives, and therefore the derivative indexes $\alpha_i$ remain unchanged as in the partitions of $\alpha$ into $\{\alpha_i\}$. Next, introduce $\gamma_i$ as stated in the lemma. We note that employing the estimates \eqref{corollary 8 -eq2}, \eqref{corollary 8 -eq3}, and \eqref{corollary 8 -eq4} are to ensure that the total number of derivatives does not exceed $|\a| + |\b|$. Moreover, since each estimate increases the number of $v$-derivatives by at most one, this leads to the definition $\gamma_i = |\beta_i|+1$, except for the top derivatives. In addition, examining all of our estimates reveals that \eqref{corollary 8 -eq2}, \eqref{corollary 8 -eq3}, and \eqref{corollary 8 -eq4} do not increase the derivative counts of $g_1$, $g_2$, and $g_3$, respectively. Collecting \eqref{corollary 8 -eq2}, \eqref{corollary 8 -eq3} and \eqref{corollary 8 -eq4}, together with \eqref{vdnorm}, 
we obtain 
\begin{equation} \label{corollary 8 -eq5}
\begin{aligned}
&\norm{\int_{\R^d} \inn{v}^{-2\th} (\nabla_{v} + v)h \cdot \pt^{\a}_{\b}\left(\mathcal{B}(\nabla_{v} F_{g_{1}})[g_{2}] \nabla_{v} g_{3}\right) \dd v} \\
&\lesssim \norm{h}_{0,0,\th,D} \sum \norm{g_{3}}_{\norm{\a_{3}},\norm{\g_{3}},\th,D} \norm{g_{2}}_{\norm{\a_{2}},\norm{\g_{2}},r,D} \prod_{j}(1+\norm{g_{1}}_{\norm{\a_{1,j}},\norm{\g_{1,j}},r,D}),
\end{aligned}
\end{equation}
for any $\alpha,\beta$, as desired, upon recalling $\norm{f}^2_{D} = \norm{v f }^2_{A} + \norm{\nabla_{v} f}^2_{A} 
$. This yields the lemma for the nonlinear term $N_1$. The term $N_{2}$ can be estimated in the same manner as $N_{1}$. For $N_{3}$ and $N_{4}$, we can bound them by substituting $\mu$ into $g_{2}$ and $g_{3}$ respectively, applying Lemma \ref{lemma:S - upper bound} instead of Lemma \ref{lemma:T - upper bound}, and then following a similar procedure done above for $N_{1}$. Combining all these bounds completes the proof of the lemma.
\end{proof}

\subsection{Proof of Proposition \ref{proposition:N upper bound}}

We are now ready to prove Proposition \ref{proposition:N upper bound}, providing the nonlinear estimates in the weighted spaces. Indeed, using \eqref{cor 7 - inequality 1}, we have  
\begin{equation} \label{corollary 8 -eq6}
\begin{aligned}
&\norm{\iint_{\T^d \times \R^d} \inn{v}^{-2\th} h\pt^{\a}_{\b}N(g_{1},g_{2},g_{3}) \;\dd x\dd v } 
\\
&\lesssim \int_{\T^d} \norm{h}_{0,0,\th,D} \sum \Big(\norm{g_{3}}_{\norm{\a_{3}},\norm{\g_{3}},\th,D} \norm{g_{2}}_{\norm{\a_{2}},\norm{\g_{2}},r,D} \prod_{j}(1+\norm{g_{1}}_{\norm{\a_{1,j}},\norm{\g_{1,j}},r,D}) \\
&\quad+\norm{g_{3}}_{\norm{\a_{3}},\norm{\g_{3}},\th,D} \norm{g_{1}}_{\norm{\a_{1,1}},\norm{\g_{1,1}},r,D}  \prod_{j}(1+\norm{g_{1}}_{\norm{\a_{1,j}},\norm{\g_{1,j}},r,D}) \\
&\quad+ \norm{g_{2}}_{\norm{\a_{2}},\norm{\g_{2}},r,D} \norm{g_{1}}_{\norm{\a_{1,1}},\norm{\g_{1,1}},r,D} \prod_{j}(1+\norm{g_{1}}_{\norm{\a_{1,j}},\norm{\g_{1,j}},r,D})\Big) \; \dd x.
\end{aligned}
\end{equation}
It remains to bound the integral in $x$. We shall apply Holder's inequality to deduce $L^2$ bounds on the top derivatives, and $L^\infty$ bounds on the terms with few spatial derivatives, which can then be bounded by $H^{d/2+\delta}$ norms due to the Sobolev embeddings, leading to the introduction of $\sigma_i = |\alpha_i|+\lfloor \frac{d}{2}+1\rfloor$. Specifically, we first focus on the case when $\norm{\a} + \norm{\b} \ge \frac{d}{2} + 5$. If $|\a_{1,j}| + |\b_{1,j}| < \frac{d}{2} +5$ for all $j$ and $|\a_{2}| + |\b_{2}| < \frac{d}{2} +5$ then, we apply the $L^2$ norm to the $h$, $g_3$ and the $L^\infty$ norm to the remaining terms then we have $\sigma_{3} = \a_{3}$ and $(|\g_{3}| = |\b_{3}|)$ or ($ |\g_{3}|=1$ and $ |\a_{3}|=0$). For $\sigma_{1,j}$ and $\sigma_{2}$, we have
\begin{align*}
|\sigma_{i}| + |\g_{i}| \le |\a_{i}|+\lfloor\frac{d}{2}+1 \rfloor + |\b_{i}|+1 \le d+6.
\end{align*}
On the other hand, if $|\a_{1,j}| + |\b_{1,j}| \ge \frac{d}{2} +5$ for some $j$, or $|\a_{2}| + |\b_{2}| \ge \frac{d}{2} +5$ then, we apply the $L^2$ norm to the $h$ $g_{1,j}$ or $g_{2}$ and the $L^\infty$ norm to the remaining terms. 
Finally, in the case when $\norm{\a} + \norm{\b} < \frac{d}{2} + 5$, we apply \eqref{T - upper 1} and \eqref{S - upper 1} regardless of how the derivatives are distributed. Then, applying the Sobolev inequality to $g_1$ and $g_2$ as in the previous step, we conclude the proof of the proposition.

\section{Energy Estimates}\label{sec-L2}

In this section, we shall derive weighted $L^2$ energy estimates and estimates on $\P f$. According to Lemma \ref{lemma:A,D - norm}, see \eqref{ineq:Pfchart}, we obtain $\norm{\inn{v}^{-r}\pt_{\b}^{\a}\P f}_{L^{p}_{v}} \approx \norm{\pt^{\a}a[f]} + \norm{\pt^{\a}b[f]}+ \norm{\pt^{\a}c[f]}$, which implies that to derive an estimate for $\P f$, it is sufficient to control only $a, b$, and $c$.

\subsection{Projection Estimates}

Let us first examine the properties of $\mathbf{P} f$ and then introduce the Burnett functions, both of which are essential for the proof. Indeed, let the Burnett functions be defined by
\begin{equation}\label{0-Burnette-function}
\begin{aligned}
&A_{ij}(v) := \Big(v_{i}v_{j} - \frac{\delta_{ij}}{d}\norm{v}^2 \Big)\sqrt{\mu},
\qquad
B_{i}(v) := v_{i}\frac{2\norm{v}^2-(d+2)}{\sqrt{d+2}} \sqrt{\mu},\quad  i,j=1,\cdots d.
\end{aligned}
\end{equation}
For each $i,j=1,\cdots, d$,  $A_{ij}(v)$ and $B_{i}(v)$ are orthogonal to every basis element $\chi_{k} $ of $\ker L$ defined as in \eqref{base-hat-chi}: namely, 
\begin{equation}\label{Burnette-orthogonal}
\int_{\R^d} \chi_{k}(v) A_{ij}(v) \dd v=0, \quad
\int_{\R^d} \chi_{k}(v) B_{i}(v) \dd v=0, \quad k=0,\cdots, d+1.
\end{equation}
The following lemma describes the macroscopic part $\P f$ and shows that it is governed by the $v$-independent functions $a, b,$ and $c$.

\begin{proposition} \label{proposition:0-macro-L2-estimate} Let $f$ be a solution of 
\begin{equation} \label{LB equation f g form}
    \pt_t f + v \cdot \nabla_{x} f = L[f] + g
\end{equation}
that satisfies the mass, momentum and energy conservations 
\begin{equation} \label{0-f-a-conservation-law}
\begin{aligned}
\iint_{\T^d\times\R^d} \begin{pmatrix}
    1 \\
    v_{i} \\ 
    \norm{v}^2
\end{pmatrix}\sqrt{\mu}
f(t,x,v) \dd v \dd x =& 0 \;\;  \text{ for all }  i\in\{1,\cdots d\}, \; \text{and}\; t\in [0, T],
\end{aligned}
\end{equation}
with $0<T\leq \infty$. In addition, assume that there is a norm $\normmb{\cdot}$ so that $g$ in \eqref{LB equation f g form} satisfies
\begin{equation} \label{ineq:prop12 assumption}
\iint_{\T^d \times \R^d} \psi \pt^{\a} g \dd v \dd x \lesssim \normm{ \psi}_{L^2_{x,v}} \normmb{\pt^{\a}g},
\end{equation}
for any test function $\psi$, and 
\begin{equation*}
\begin{aligned}
\iint_{\T^d\times\R^d}\begin{pmatrix}
    1 \\
    v_{i} \\ 
    \norm{v}^2
\end{pmatrix}\sqrt{\mu}
g \dd v \dd x =& 0 \;\;  \text{ for all }  i\in\{1,\cdots d\}.
\end{aligned}
\end{equation*}
Then,  for all $0\leq s\leq t \leq T$, the following estimates hold:
\begin{align} \label{0-P-f-macro-L2}
\int_{s}^{t} \normm{ \P \pt^{\a} f}_{L^{2}_{x,v}}^2\dd \tau \lesssim \; 
\big[ G_{0} (t)-  G_{0} (s)\big]
+  \int_{s}^{t}   \Big( \normm{\ip \pt^{\a} f}_{D}^{2} + \normmb{\pt^{\a}g}\Big) \dd \tau, 
\end{align}
where $|G_{0}(t)| \lesssim \normm{\pt^{\a} f(t)}_{L^{2}_{x,v}}^2$.
\end{proposition}

\begin{proof}
The proof employs the test function method in \cite{esposito2013non, Esposito2017} combined with the elliptic theory. Indeed, taking spacial derivatives to \eqref{LB equation f g form} we have
\begin{equation} \label{LB equation f derivative form}
    \pt_t \pt^{\a} f + v \cdot \nabla_{x} \pt^{\a} f = \pt^{\a} L[f] + \pt^{\a}g
\end{equation}
Multiplying \eqref{LB equation f derivative form} by a test function $\psi_{p}$ yields the weak formulation
\begin{equation}\label{0-test-equation-uniform-form}
\begin{aligned}
\underbrace{\iint_{\T^{d} \times \R^{d}}  \psi_{p} \pt_t \pt^{\a} f \dd v \dd x}_{:=\Xi_{p}^{1}}
- \underbrace{\iint_{\T^{d} \times \R^{d}}  (v \cdot \nabla_x \psi_{p}) \pt^{\a} f \dd v \dd x}_{:=\Xi_{p}^{2}}
%----------------
=  \underbrace{\iint_{\T^{d} \times \R^{d}} \Big[\psi_{p} \pt^{\a} L[f] + \psi_{p} \pt^{\a} g\Big] \dd v \dd x}_{:=\Xi_{p}^{3}}.
\end{aligned}
\end{equation}
Here the index $p\in \{a,b,c\}$ marks estimates of $a, b$ and $c$. To estimate $\P f$, by the representation \eqref{Pf-abc-hat}, it suffices to estimate $a, b$ and $c$ by Lemma \ref{lemma:A,D - norm}, see \eqref{ineq:Pfchart}.

\subsection*{Step 1)  Estimate for $\int_{s}^{t}\|b\|_{L^{2}_{x}}\dd \tau$.}

In \eqref{0-test-equation-uniform-form} we choose the test function
\begin{equation}\label{0-psi-b-2-definition}
\begin{aligned}
 \psi_{b} := & \sum_{i,j=1}^{d} \pt_{j} \varphi_{b,i} A_{ij}(v) - \sum_{i=1}^{d} \pt_{i}\varphi_{b,i}\chi_{d+1}(v) \frac{d-2}{2\sqrt{2d}}
 \\
=&\sum_{i,j=1}^{d} \pt_{j} \varphi_{b,i} v_iv_j\sqrt{\mu}
 -\sum_{i=1}^{d} \pt_{i} \varphi_{b,i}\frac{2|v|^2-(d-2)}{4}\sqrt{\mu},
\end{aligned}
\end{equation}
where the vector-valued function $\varphi_{b}$ satisfies the elliptic equation
\begin{equation}\label{0-b-2-elliptic-equation}
\begin{aligned}
- \Delta_x \varphi_{b} = \pt^{\a} b \;\;  &\text{in } \T^d, \quad \int_{\T^d} \varphi_{b} \dd x = 0.
\end{aligned}
\end{equation}
The standard elliptic theory (e.g., \cite{gilbarg1977elliptic}) guarantees a unique solution of \eqref{0-b-2-elliptic-equation} with the estimate
\begin{equation}\label{0-b-2-elliptic-estimate}
\begin{aligned}
\normm{\nabla^2_x \varphi_{b} }_{L^{2}_{x}} + \normm{\nabla_x \varphi_{b} }_{L^{2}_{x}} + \normm{\varphi_{b} }_{L^{2}_{x}} & \lesssim \normm{\pt^{\a} b}_{L^{2}_{x}}.
\end{aligned}
\end{equation}
We now estimate each term $ \Xi_{b}^{k}$ in \eqref{0-test-equation-uniform-form}, with $\psi_p = \psi_b$.
For $ \Xi_{b}^{1}$, integration by parts yields
\begin{equation*}
\begin{aligned}
\int_{s}^{t}  \Xi_{b}^{1}\dd \tau = &\; \big[ G_{b} (t) - G_{b} (s)\big] - \int_{s}^{t}\iint_{\T^d \times \R^{d}} \pt_t\psi_{b} \pt^{\a} f \dd v \dd x \dd \tau,
\end{aligned}
\end{equation*}
where $G_{b}$ is $\iint_{\T^d \times \R^{d}} \psi_{b} \pt^{\a} f \dd v \dd x$. Using the elliptic estimates \eqref{0-b-2-elliptic-estimate}, we bound 
\begin{equation}\label{bdGb}
\normm{G_{b}}_{L^2_{x,v}} \lesssim \normm{\psi_{b}}_{L^2_{x,v}} \normm{\pt^{\a} f}_{L^2_{x,v}} \lesssim \normm{\nabla \varphi_{b}}_{L^2_{x}} \normm{\pt^{\a} f}_{L^2_{x,v}} \lesssim \normm{\pt^{\a} f}_{L^2_{x,v}}^2.
\end{equation}
The contributions from $a$ and $b$ vanish due to \eqref{Burnette-orthogonal} and the identity $\int_{\R^{d}}\chi_{d+1}\pt^{\a} f\dd v= \pt^{\a} c$.
Using \eqref{0-b-2-elliptic-estimate} and Corollary \ref{corollary:Upper bound of L}, we obtain 
\begin{equation} \label{0-Theta1 - b estimate}
\begin{aligned}
 \Big| \int_{s}^{t}  \Xi_{b}^{1} \Big|
 \leq  &    \big[ G_{b} (t) - G_{b}(s) \big] +  \int_{s}^{t} \normm{\pt_t \nabla_x \varphi_{b}}_{L^2_x} \big( \normm{\pt^{\a} c}_{L^2_x} + \normm{\ip \pt^{\a} f}_{D} \big),
\end{aligned}
\end{equation}
where $G_b$ satisfies \eqref{bdGb} as claimed. 

To handle $\Xi_{b}^{2}$, we use the expression in the second line of \eqref{0-psi-b-2-definition} and split
\begin{equation}\label{0-Theta3 - b estimate-psi-b-2}
\begin{aligned}
- v\cdot \nabla_x  \psi_{b}
=&- \sum_{i,j,k=1}^{d}  \pt_{j} \pt_{k} \varphi_{b,i}
\P \left (v_i v_j v_k \sqrt{\mu} \right)
 +  \sum_{i,k=1}^{d} \pt_{i} \pt_k  \varphi_{b,i}  v_k \frac{2|v|^2-(d-2)}{4} \sqrt{\mu}\\
%---------------
%---------------
& - \sum_{i,j,k=1}^{d}  \pt_{j} \pt_{k} \varphi_{b,i}
\ip \left (v_i v_j v_k \sqrt{\mu} \right)\\
:=& {K}_1+{K}_2+{K}_3.
\end{aligned}
\end{equation}
A direct calculation yields
\begin{equation}\label{b-estimate-K1-final-form}
\begin{aligned}
{K}_1
= & -\sum_{l=1}^d v_l\sqrt{\mu} \Big( \sum_{i=j=k=l} + \sum_{i =  j \neq k=l} + \sum_{i = k \neq j=l} + \sum_{i = l \neq j=k}  \Big)
\pt_{j} \pt_{k}{\varphi}_{b,i} \int_{\mathbb{R}^d}v_i v_j v_k v_l \mu \dd v\\
= & -\frac{1}{4}\sum_{l=1}^d v_l\sqrt{\mu} \Big(3\pt_{l} \pt_{l}{\varphi}_{b,l} + 2\sum_{i\neq l} \pt_{i} \pt_{l}{\varphi}_{b,i}
+ \sum_{i\neq l} \pt_{i} \pt_{i}{\varphi}_{b,l}\Big) \\
= & -\frac{1}{4}\sum_{l=1}^d v_l\sqrt{\mu} \Big( 2\sum_{i =1}^{d} \pt_{i} \pt_{l}{\varphi}_{b,i}
+ \Delta {\varphi}_{b,l}\Big),
\end{aligned}
\end{equation}
where  we used the identity
\begin{equation*}
\begin{aligned}
&\int_{\mathbb{R}^d}v_i^2 v_j^2\mu \dd v=
 \left\{
   \begin{array}{ll}
     \frac{3}{4}, & \hbox{if $i=j$,} \\
     \frac{1}{4}, & \hbox{if $i\neq j$.}
   \end{array}
 \right.
\end{aligned}
\end{equation*}
Moreover, using $\int_{\mathbb{R}^d}v_i^2(2|v|^2-(d-2)) \mu \dd v=2$, we obtain
\begin{equation}\label{0-pf-b-left2}
\begin{aligned}
 \iint_{\T^d \times\R^d} {K}_2  \mathbf{P}f
= & \int_{\T^d} \frac{1}{2}\sum_{l =1}^{d}\sum_{i =1}^{d}\pt^{\a} b_{l} \pt_{i} \pt_{l}{\varphi}_{b,i} ,
\end{aligned}
\end{equation}
Combining \eqref{0-Theta3 - b estimate-psi-b-2}, \eqref{b-estimate-K1-final-form} and \eqref{0-pf-b-left2}, we find
\begin{equation}\label{0-Theta3 - b estimate}
\begin{aligned}
{\Xi}_{b}^{2}
=&\iint_{\T^{d}\times\mathbb{R}^d} ({K}_1 + {K}_2)  \mathbf{P}\pt^{\a} f \dd v \dd x+ E_{b}\\
=&-\frac{1}{4}\sum_{l=1}^d\int_{\T^d}\pt^{\a} {b}_l\Delta{\varphi}_{b,l} \dd x+ E_{b}  =\frac{1}{4}\|{\pt^{\a} b}\|_{L^2_x}^2+ E_{b},
\end{aligned}
\end{equation}
 where we used \eqref{0-b-2-elliptic-equation}. By \eqref{0-b-2-elliptic-estimate} and Corollary \ref{corollary:Upper bound of L}, the remainder satisfies
\begin{equation}\label{0-E-b-2-estimate-L2}
\begin{aligned}
 |E_{b}|=&
 \Big | \iint_{\T^d\times\mathbb{R}^d} {K}_3 \ip \pt^{\a} f \dd v \dd x\Big |
 \lesssim   \normm{\ip \pt^{\a} f}_{D} \normm{\pt^{\a} {b}}_{L^{2}_{x}},
 \end{aligned}
\end{equation}
for any $r\ge 0$. Finally, for $\Xi_{b}^{3}$, Corollary \ref{corollary:Upper bound of L} and \eqref{ineq:prop12 assumption} give
\begin{equation} \label{0-Theta4-b2-estimate}
\begin{aligned}
\norm{\Xi_{b}^{3}}
\lesssim   \normm{\pt^{\a} b}_{L^{2}_{x}} \Big(\normm{\ip \pt^{\a} f}_{D} + \normmb{\pt^{\a}g} \Big).
\end{aligned}
\end{equation}

It remains to estimate $ \normm{\pt_t \nabla_x
{\varphi}_{b}}_{L^2_x} $ appearing in \eqref{0-Theta1 - b estimate}. For this,  we choose
${\psi}_{b} = \pt_t{\varphi}_{b} \cdot v \sqrt{ \mu }$ in \eqref{0-test-equation-uniform-form}  and denote the resulting three terms by $\hat{\Xi}_{b}^{k}$ ($k=1,2,3$). Clearly, $\hat{\Xi}_{b}^{3}=0$. Using \eqref{0-b-2-elliptic-equation}, we obtain
\begin{equation} \label{0-b-t-theta1-estimate}
\begin{aligned}
\hat{\Xi}_{b}^{1}
=  \int_{\T^d} \pt_t {\varphi}_{b}\cdot  \pt_t \pt^{\a}{b}  \dd x
=  - \int_{\T^d} \pt_t {\varphi}_{b}\cdot  \Delta_x \pt_t {\varphi}_{b}  \dd x
=  \normm{\nabla_x \pt_t {\varphi}_{b}}_{L^2_{x}}^2.
\end{aligned}
\end{equation}
A Poincar\'{e}'s inequality and Corollary \ref{corollary:Upper bound of L} yield the following inequality.
\begin{equation} \label{0-b-t-theta3-estimate}
\begin{aligned}
\norm{ \hat{\Xi}_{b}^{2}}
 \lesssim & \normm{\nabla_x \pt_t{\varphi}_{b}}_{L^2_{x}} \big(\normm{\pt^{\a}{a}}_{L^2_x} + \normm{\pt^{\a}{c}}_{L^2_x} + \normm{\ip \pt^{\a} f}_{D}  \big).
\end{aligned}
\end{equation}
Collecting \eqref{0-test-equation-uniform-form} and \eqref{0-b-t-theta1-estimate}--\eqref{0-b-t-theta3-estimate} yields
\begin{equation}\label{0-varphib - pt t estimate}
\begin{aligned}
\normm{\nabla_x \pt_t {\varphi}_{b} }_{L^2_{x}} \lesssim& \normm{\pt^{\a}{a}}_{L^2_{x}} + \normm{\pt^{\a} {c}}_{L^2_{x}}   + \normm{\ip \pt^{\a} f}_{D}.
\end{aligned}
\end{equation}
Finally,  inserting \eqref{0-varphib - pt t estimate} into \eqref{0-Theta1 - b estimate} and  combining it with \eqref{0-test-equation-uniform-form} and \eqref{0-Theta3 - b estimate}--\eqref{0-Theta4-b2-estimate}, we obtain
\begin{equation} \label{0-tildeb - l2 estimate final}
\begin{aligned}
\int_{s}^{t} \normm{\pt^{\a} {b}}_{L^2_{x}}^2 
&\leq C_{b} \Big \{ 
 \big[ G_{b} (t)-   G_{b} (s)\big]
 +  \delta_b \int_{s}^{t} \normm{\pt^{\a} {a}}_{L^2_x}^2\\
&\quad
 +  \int_{s}^{t}   \big ( \normm{\pt^{\a} {c}}_{L^2_x}^2+ \normm{\ip \pt^{\a} f}_{D}^{2} +  \normmb{\pt^{\a}g}^2  \Big \},
\end{aligned}
\end{equation}
where the small constant $\delta_b>0$ arises from H\"{o}lder's inequality.

\subsection*{Step 2)  Estimates for $\int_{s}^{t}\|a\|_{L^{2}_{x}}\dd \tau$.}

In \eqref{0-test-equation-uniform-form}, we consider the test function
\begin{equation} \label{0-psi-a-2-definition}
  \psi_{a} := \sum_{i=1}^{d} \pt_{i} \varphi_{a} \big[\frac{\sqrt{d+2}}{2} B_{i}(v) - \frac{d+2}{2\sqrt{2}}\chi_{i}(v)\big] =  \sum_{i=1}^{d} \pt_{i} \varphi_{a}  v_i(|v|^2 - (d+2))\sqrt{\mu},
\end{equation}
  where $\varphi_{a}$ satisfy the elliptic equations
\begin{align}
- \Delta_x \varphi_{a} = \pt^{\a}a \;\text{ in }\T^d,\qquad \int_{\T^d} \varphi_{a} \dd x =0, \label{0-a-2-elliptic-equation}
\end{align}
According to the standard elliptic theory (cf. \cite{gilbarg1977elliptic}), $\varphi_{a}$ satisfies
\begin{align}
\normm{\nabla^2 \varphi_{a}}_{L^2_x} + \normm{\nabla \varphi_{a}}_{L^2_x} + \normm{\varphi_{a}}_{L^2_x} &\lesssim \normm{a}_{L^2_x}, \label{0-a-2-elliptic-estimate}
\end{align}
We now estimate each term $ \Xi_{a}^{k}$ in \eqref{0-test-equation-uniform-form}.
For $ \Xi_{a}^{1}$, integration by parts similar to \eqref{0-Theta1 - b estimate} yields
\begin{equation} \label{0-Theta1 - a estimate}
\begin{aligned}
\int_{s}^{t}   \norm{ \Xi_{a}^{1}  }
 \leq  &   \left[G_{a} (t)-  G_{a} (s)\right] +  \int_{s}^{t}\normm{\pt_t \nabla_x \varphi_{a,2}}_{L^2_x}  \big( \normm{\pt^{\a} b}_{L^2_x} + \normm{\ip \pt^{\a} f}_{D}\big),
\end{aligned}
\end{equation}
where the contributions from $a$ and $c$ again vanish due to the orthogonality relation.

To handle $\Xi_{a}^{2}$, we use the last expression in \eqref{0-psi-a-2-definition} and split
\begin{equation}\label{0-Theta3 - a estimate-psi-a-2}
\begin{aligned}
- v\cdot \nabla_x  \psi_{a}
=&- \sum_{i,j=1}^{d}  \pt_{i}\pt_{j} \varphi_{a}
\P \left (v_i v_j \norm{v}^2 \sqrt{\mu} \right)
 +  (d+2)\sum_{i,j=1}^{d} \pt_{i} \pt_j  \varphi_{a} \P \left( v_{i} v_{j} \sqrt{\mu}\right)\\
%---------------
%---------------
& - \sum_{i,j=1}^{d}  \pt_{i} \pt_{j} \varphi_{a}
\ip \left (v_i v_j (\norm{v}^2 -  (d+2)) \right).
\end{aligned}
\end{equation}
A direct calculation yields
\begin{equation}\label{ac-estimate-K1-final-form}
\begin{aligned}
\inn{v_{i}v_{j}\norm{v}^2 \sqrt{\mu},\chi_{0}}=\frac{d+2}{4}\d_{ij}, \quad & \inn{v_{i}v_{j}\norm{v}^2 \sqrt{\mu},\chi_{k}}=0, \quad \inn{v_{i}v_{j}\norm{v}^2 \sqrt{\mu},\chi_{d+1}}=\frac{d+2}{\sqrt{2d}}\d_{ij}, \\
\inn{v_{i}v_{j} \sqrt{\mu},\chi_{0}}= \frac{1}{2}\d_{ij}, \quad & \inn{v_{i}v_{j} \sqrt{\mu},\chi_{k}}=0, \quad \inn{v_{i}v_{j} \sqrt{\mu},\chi_{d+1}}=\frac{1}{\sqrt{2d}}\d_{ij},
\end{aligned}
\end{equation}
for $k=1, \cdots d$.
Combining \eqref{ac-estimate-K1-final-form} and \eqref{0-Theta3 - a estimate-psi-a-2}, we have
\begin{align}
&\Xi_{a}^{2}
=  \int_{\T^d} \frac{3(d+2)}{4}\Delta_x \varphi_{a,2} \pt^{\a} a \dd x   + E_{a}
= -\frac{3(d+2)}{4} \normm{ \pt^{\a} a }^2_{L^{2}_{x}}  + E_{a}, \label{0-Theta3 - a estimate}
\end{align}
where $E_{a}$ arises from the $\ip \pt^{\a} f$ part and can be controlled similarly to what was done in \eqref{0-E-b-2-estimate-L2}.
The terms $\Xi_{a}^{3}$ are estimated analogously to \eqref{0-Theta4-b2-estimate}.

To estimate $ \normm{\pt_t \nabla_x \varphi_{a}}_{L^2_x}$ appearing in \eqref{0-Theta1 - a estimate},  we choose $\psi_{a} = \pt_t\varphi_{a} \sqrt{\mu}$ in \eqref{0-test-equation-uniform-form}. Arguing similarly to the derivation of \eqref{0-varphib - pt t estimate} and Poincar\'{e}'s inequality, we obtain
\begin{align}
& \normm{\nabla_x \pt_t \varphi_{a} }_{L^2_x} \lesssim  \normm{\pt^{\a} b}_{L^2_x}.  \label{0-varphia - pt t estimate}
\end{align}
Inserting \eqref{0-varphia - pt t estimate} into \eqref{0-Theta1 - a estimate} and  combining it with \eqref{0-test-equation-uniform-form} and the estimates for $\Xi_{a,2}^{3}$, we obtain 
\begin{equation} \label{0-tildea - l2 estimate final}
\begin{aligned}
\int_{s}^{t} \normm{a}_{L^2_{x}}^2  \leq
 C_{a} \Big \{
 G_{a} (t)-  G_{a} (s) +  \int_{s}^{t}   \big ( \normm{\pt^{\a} {b}}_{L^2_x}^2+ \normm{\ip \pt^{\a} f}_{D}^{2} +  \normmb{\pt^{\a}g}^2 \Big \}.
\end{aligned}
\end{equation}

\subsection*{Step 3)  Estimate for $\int_{s}^{t}\|{c}\|_{L^{2}_{x}}\dd \tau$.}

Finally, we derive the $L^2$ estimates on $c$. Similarly as done above, in \eqref{0-test-equation-uniform-form}, we choose
\begin{equation}\label{0-psi-b-n-definition}
\begin{aligned}
 {\psi}_{c}(t,x,v) := &  \sum_{i=1}^{d} \pt_{i} {\varphi}_{c}(t,x) {B}_{i}(v),
\end{aligned}
\end{equation}
where ${\varphi}_{c}(x)$ satisfy the elliptic equations
\begin{align}
&- \Delta_x {\varphi}_{c} = \pt^{\a}{c} \;\;
\text{in } \T^d, \quad 
\int_{\T^d} {\varphi}_{c} \dd x =0. \label{0-c-2-elliptic-equation}
\end{align}
These solutions satisfy elliptic estimates analogous to \eqref{0-a-2-elliptic-estimate}. As a result, for $ \Xi_{c}^{1}$, we obtain, similar to \eqref{0-Theta1 - a estimate},
\begin{equation} \label{0-Theta1 - c estimate}
\begin{aligned}
\Big | \int_{s}^{t}    \Xi_{c}^{1}  \dd \tau \Big|
\lesssim  &\;  \big[G_{c} (t)-  G_{c} (s)\big] + \int_{s}^{t}\normm{\pt_t \nabla_x \varphi_{c}}_{L^2_x}  \normm{\ip \pt^{\a} f}_{D},
\end{aligned}
\end{equation}
where the contribution from ${\P}f$ vanishes due to \eqref{Burnette-orthogonal}. Similar to $\Xi_{a}^{2}$, by \eqref{ac-estimate-K1-final-form} we have
\begin{equation}
\begin{aligned}
\Xi_{c}^{2}
= & -\frac{\sqrt{d+2}}{\sqrt{2d}}\int_{\T^d} \pt^{\a}c \Delta_x \varphi_{c}  + E_{c}
= \frac{\sqrt{d+2}}{\sqrt{2d}} \normm{ \pt^{\a}c }^2_{L^{2}_{x}}  + E_{c},
\end{aligned}
\end{equation}
where the remainder $E_{c}$ is due to $\ip f$, while the terms $\Xi_{c}^{3}$ is estimated similarly to \eqref{0-Theta4-b2-estimate}. Finally, to bound $ \normm{\pt_t \nabla_x \varphi_{c}}_{L^2_x} $ in \eqref{0-Theta1 - c estimate}, we choose $\psi_{c} = \pt_t\varphi_{c}  \chi_4(v)$ in \eqref{0-test-equation-uniform-form}. Inserting $\psi_{c}$ into \eqref{0-Theta1 - c estimate} and  combining it with \eqref{0-test-equation-uniform-form} and the estimates for $\Xi_{c}^{2}, \Xi_{c}^{3}$, we obtain
\begin{equation} \label{0-tildec - l2 estimate final}
\begin{aligned}
\int_{s}^{t} \normm{c}_{L^2_{x}}^2  \leq
 C_{c} \Big \{
 G_{c} (t)-  G_{c} (s) +  \int_{s}^{t}   \big ( \d_{c}\normm{\pt^{\a} {b}}_{L^2_x}^2+ \normm{\ip \pt^{\a} f}_{D}^{2} +  \normmb{\pt^{\a}g}^2  \Big \},
\end{aligned}
\end{equation}
where the small constant $\delta_c>0$   arises from Young's inequality.

\subsection*{Step 4)  Combination of the estimates for $a$,  $b$ and  $c$.}

Choose $\delta_{b} = (8 C_{a} C_{b})^{-1}$ and $\delta_{c} = (8 C_{b} C_{c})^{-1}$.
A direct computation of
$$
(8C_{a}C_{b})^{-1} \times \eqref{0-tildea - l2 estimate final} + (2 C_{b})^{-1}  \times\eqref{0-tildeb - l2 estimate final} +\eqref{0-tildec - l2 estimate final}
$$
yields \eqref{0-P-f-macro-L2}, upon recalling \eqref{ineq:Pfchart}. This ends the proof of the proposition. 
\end{proof}

\subsection{$L^2$ Estimates}

To obtain an energy estimate for the nonlinear equation, we again consider the following linear equation:
\begin{equation} \label{linear eq f - lemma10}
\begin{cases}
&\pt_t f + v \cdot \nabla_{x} f -L[f] = g\\    
&f|_{t=0} = f_{0}.
\end{cases}
\end{equation}

\begin{lemma} \label{weighted energy estimate} (Weighted energy estimate)
Let $f$ be a solution of \eqref{linear eq f - lemma10}. Then, the following energy estimates hold for $\th \in \R$ and any multi-index $\a$ and $\b$.
\begin{enumerate}

    \item Unweighted energy estimate for $\pt^\a f$

    \begin{equation*}
\begin{aligned}
\pt_{t} \normm{\pt^\a f}^2_{L^2_{x,v}} + \normm{\pt^{\a} \ip f}_{D}^2
\lesssim \int_{\T^d \times \R^d} \pt^{\a} f \pt^{\a} g \dd x \dd v.
\end{aligned}
\end{equation*}

    \item Weighted energy estimate for $\pt^\a f$

\begin{equation*}
\begin{aligned}
\pt_{t} \normm{\inn{v}^{-\th}\pt^\a f}^2_{L^2_{x,v}} + \normm{\inn{v}^{-\th} \pt^{\a} f}_{D}^2
\lesssim \normm{\pt^{\a} f}_{D}^2 + \int_{\T^d \times \R^d} \inn{v}^{-2 \th}\pt^{\a} f \pt^{\a} g \dd x \dd v.
\end{aligned}
\end{equation*}

    \item Weighted energy estimate for $\pt^\a_{\b}f$

\begin{equation*}
\begin{aligned}
\pt_{t} \normm{\inn{v}^{-\th}\pt^\a_{\b} f}^2_{L^2_{x,v}} + \normm{\inn{v}^{-\th} \pt^\a_{\b} f}_{D}^2 
\lesssim&  \normm{\inn{v}^{-\th+1} \pt^{\a+1}_{\b-1} f}_{D}^2+ \sum_{\norm{\b'} < \norm{\b}}\normm{\inn{v}^{-\th} \pt^{\a}_{\b'} f}_{D}^2
\\&+ \int_{\T^d \times \R^d} \inn{v}^{-2 \th}\pt^\a_{\b} f \pt^\a_{\b} g \dd x \dd v.
\end{aligned}
\end{equation*}

\end{enumerate}
\end{lemma}

\begin{proof}
Applying $\partial^{\alpha}$ to \eqref{linear eq f - lemma10}, we obtain
\begin{equation}\label{pta f energy estimate}
\pt_{t} \pt^{\a} f + v\cdot \nabla_{x} \pt^{\a} f  - \pt^{\a} L[f] = \pt^{\a} g.
\end{equation}
Multiplying $\inn{v}^{-2\th} \pt^{\a} f$ and integrating the identity over the phase space,  we have
\begin{equation*}
\begin{aligned}
\pt_{t} \normm{\pt^\a f}^2_{L^2_{x,v}} + \int_{\T^d \times \R^d}{\inn{v}^{-2\th} \pt^\a f L\pt^\a f} \dd x\dd v
\lesssim \int_{\T^d \times \R^d} \inn{v}^{-2\th}\pt^{\a} f \pt^{\a} g \dd x \dd v.
\end{aligned}
\end{equation*}
Note that $L$ commutes with $\partial^{\alpha}$. Therefore, by applying Lemma \ref{lemma:Coerceivity of L}, the first two statements of Lemma \ref{weighted energy estimate} follow. As for the last statement, we apply  $\partial^{\alpha}_{\beta}$ to \eqref{linear eq f - lemma10} and obtain
\begin{equation} \label{ptab f energy estimate}
\pt_{t} \pt^{\a}_{\b} f + v\cdot \nabla_{x} \pt^{\a}_{\b} f + \sum_{\substack{e_{i} \le \b\\ \norm{e_{i}}=1
}}\pt^{\a+e_{i}}_{\b-e_{i}} f - \pt^{\a}_{\b} L[f] = \pt^{\a}_{\b} g.
\end{equation}
We multiply \eqref{ptab f energy estimate} by $\inn{v}^{-2 \th} \partial^{\a}_{\b} f$ and integrate over $v$ and $x$. Applying again Lemma \ref{lemma:Coerceivity of L}, we obtain
\begin{equation*}
\begin{aligned}
\pt_{t} \normm{\inn{v}^{-\th}\pt^\a_{\b} f}^2_{L^2_{x,v}} + \normm{\inn{v}^{-\th} \pt^\a_{\b} f}_{D}^2  
&\lesssim   \normm{\inn{v}^{-1/2}\inn{v}^{-\th}\pt^\a_{\b} f}_{L^2_{x,v}} \normm{\inn{v}^{-1/2} \inn{v}^{-\th+1} \pt^{\a+1}_{\b-1} f}_{L^2_{x,v}}\\&\quad + \sum_{\norm{\b'} < \norm{\b}}\normm{\inn{v}^{-\th} \pt^{\a}_{\b'} f}_{D}^2
+ \int_{\T^d \times \R^d} \inn{v}^{-2 \th}\pt^\a_{\b} f \pt^\a_{\b} g \dd x \dd v,
\end{aligned}
\end{equation*}
Noting
$\normm{\inn{v}^{-1/2}f }_{L^2_{x,v}}\lesssim \normm{f}_{D}$, see \eqref{equivD}, we thus obtain the last statement as stated. This completes the proof of the lemma.  
\end{proof}

\section{Proof of the Theorem 1}\label{pf of main thm}

In this section, we prove our main results stated in Theorem \ref{theo-main}. We establish the global a priori estimate \eqref{0-uniform-bound} for small initial perturbations $  \normmm{{f}_{0}} \leq  \d_0$.
For this, assume that a solution $f$ to \eqref{LBeqs} exists on $[0, T]$ with $0 < T \leq \infty$. Recall that the bootstrap norm $\normmm{\cdot}$ in \eqref{normmm -definition} is defined by
\begin{equation}\label{defBignorm}
\begin{aligned}
 \normmm{f}(t) :=&\sup_{s\in [0,t]}\mathscr{E}_{d+7}^{\frac{1}{2}}[f](s)  + \mathscr{D}_{d+7}^{\frac{1}{2}}[f](t),
\end{aligned}
\end{equation}
where the iterative energy and dissipation functionals $\mathscr{E}_{d+7}^{\frac{1}{2}}[f](t)$ and $\mathscr{D}_{d+7}^{\frac{1}{2}}[f](t)$ are given as in \eqref{energy definition}. To proceed, we first note that 
\begin{equation} \label{d and e relation in section4}
e_{N}[\inn{v}^{-1/2}f](t) \le d_{N}[f](t)  \le e_{N}[\inn{v}^{-1/2}f](t) + e_{N+1}[\inn{v}^{1/2}f](t),
\end{equation}
for any $N\ge 0$, which follow directly from the definition in \eqref{energy definition} and the estimate \eqref{equivD}. In what follows, we fix $N=d+7$.

%Using \eqref{equivD}, we compute 
%$$
%\begin{aligned}
%{d}_{N}[f](t) 
%&= \sum_{N_{1}+N_{2} \le N}  \normm{f(t)}_{N_{1},N_{2},-N+N_{1}+2N_{2},D}^{2} 
%\\
%&\lesssim
% \sum_{N_{1}+N_{2} \le N}  \normm{\langle v\rangle^{-1/2}f(t)}_{N_{1},N_{2},-N+N_{1}+2N_{2},L^2}^{2}  
% +  \sum_{N_{1}+N_{2} \le N}  \normm{\langle v\rangle^{-1/2}\nabla_vf(t)}_{N_{1},N_{2},-N+N_{1}+2N_{2},L^2}^{2} 
%\\
%& \lesssim e_{N}[\inn{v}^{-1/2}f](t)+  \sum_{N_{1}+N_{2} \le N}  \normm{\langle v\rangle^{-1/2}f(t)}_{N_{1},N_{2}+1,-N+N_{1}+2N_{2},L^2}^{2} 
%\end{aligned}
%$$
%Changing $N_2 \mapsto \tilde N_2 = N_2+1$, we get 
%$$
% \begin{aligned}
% &
% \sum_{N_{1}+N_{2} \le N}  \normm{\langle v\rangle^{-1/2}f(t)}_{N_{1},N_{2}+1,-N+N_{1}+2N_{2},L^2}^{2} 
% \\
% &\lesssim \sum_{N_{1}+\tilde N_{2} \le N+1}  \normm{\langle v\rangle^{-1/2}f(t)}_{N_{1},\tilde N_{2},-N+N_{1}+2\tilde N_{2} - 2,L^2}^{2} 
% \\
% &\lesssim \sum_{N_{1}+\tilde N_{2} \le N+1}  \normm{\langle v\rangle^{1/2}f(t)}_{N_{1},\tilde N_{2},-(N+1)+N_{1}+2\tilde N_{2} ,L^2}^{2} 
%\\& \le e_{N+1}[\inn{v}^{1/2}f](t)
%  \end{aligned}$$
%upon noting the weight $\langle v\rangle^{-r}$ in the norm $\|\cdot \|_{N_1,N_2,r,L^2}$.  
% 

\subsection*{STEP 1) Nonlinear term control.}

In this step, we shall prove that 
\begin{equation}\label{claimbd}
\norm{\iint_{\T^d \times \R^d} \inn{v}^{-2(-N+\norm{\a}+2\norm{\b})} \pt^{\a}_{\b}h \pt^{\a}_{\b}N(f,f,f) \dd v \dd x} \lesssim d_{N}[h]d_{N}[f]e_{N}[f](1+e_{N}(f))^{N}
\end{equation}
for all multi-indexes $\alpha,\beta$ with $|\alpha|+|\beta|\le N$. Indeed, using Proposition \ref{proposition:N upper bound}, with $\theta = -N+\norm{\a}+2\norm{\b}$, we have
\begin{equation} \label{ineq:section 5 step1}
\begin{split}
&\norm{\iint_{\T^d \times \R^d} \inn{v}^{-2(-N+\norm{\a}+2\norm{\b})} \pt^{\a}_{\b}h \pt^{\a}_{\b}N(f,f,f) \dd v \dd x} 
\\&\lesssim  \normm{\pt^{\a}_{\b}h}_{0,0,-N+\norm{\a}+2\norm{\b},D} 
\Big(\sum \normm{f}_{\norm{\s_{3}},\norm{\g_{3}},-N+\norm{\a}+2\norm{\b},D}
\normm{f}_{\norm{\s_{2}},\norm{\g_{2}},r,D} \prod_{j}(1+\normm{f}_{\norm{\s_{1,j}},\norm{\g_{1,j}},r,D}) \\
&\qquad+ \normm{f}_{\norm{\s_{3}},\norm{\g_{3}},-N+\norm{\a}+2\norm{\b},D} \normm{f}_{\norm{\s_{1,1}},\norm{\g_{1,1}},r,D}  \prod_{j \neq 1}(1+\normm{f}_{\norm{\s_{1,j}},\norm{\g_{1,j}},r,D})\\
&\qquad+ \normm{f}_{\norm{\s_{2}},\norm{\g_{2}},r,D} \normm{f}_{\norm{\s_{1,1}},\norm{\g_{1,1}},r,D}  \prod_{j \neq 1}(1+\normm{f}_{\norm{\s_{1,j}},\norm{\g_{1,j}},r,D}) \Big),
\end{split}
\end{equation}
for any $r\ge 0$. We now bound each term on the right-hand side. By the definition of the dissipation norm, we have $\normm{\pt^{\a}_{\b}h}_{0,0,-N+\norm{\a}+2\norm{\b},D}\le d_{N}[h]$. It thus remains to focus on bounding the norms of $f$. We first claim that 
\begin{equation}\label{claima123}
\begin{aligned}
\normm{f}_{\norm{\sigma_{1,j}},\norm{\gamma_{1,j}},r,D} + \normm{f}_{\norm{\sigma_{2}},\norm{\gamma_{2}},r,D} + 
\normm{f}_{\norm{\sigma_{3}},\norm{\gamma_{3}},-N+\norm{\a}+2\norm{\b},D} \lesssim& d_{N}[f]
\end{aligned}
\end{equation}
for large enough $r\ge 0$. Indeed, in view of Remark \ref{rem-abi}, either $|\sigma_i|+|\gamma_i| \le d+6$ or $|\sigma_i|+|\gamma_i| \le |\alpha|+|\beta|$, both of which imply that $|\sigma_i|+|\gamma_i|\le N$, recalling $N=d+7$. Therefore, $\normm{f}_{\norm{\sigma_i},\norm{\gamma_i},r,D} \le d_{N}[f]$, upon taking $r$ large enough so that $r\ge -N + |\sigma_i|+2|\gamma_i|$. As for the norm involving $(\sigma_3, \gamma_3)$, we observe that 
in the case when $(\sigma_3, \gamma_3) = (|\alpha_3|, |\beta_3|)$, then we clearly have 
$\normm{f}_{\norm{\a_{3}},\norm{\b_{3}},-N+\norm{\a}+2\norm{\b},D} \lesssim d_{N}[f]
$ as desired. On the other hand, in the case when $(\sigma_3, \gamma_3) =(|\alpha_3|+\lfloor\frac{d}{2}+1 \rfloor , |\beta_3|)$, then by construction, see \eqref{defsigmai}, we have $\norm{\a_3}+\norm{\b_3} < \frac{d}{2}+5$ and there is an $i\in \{(1,1),\cdots (1,n),2\}$ so that 
$\norm{\a_i}+\norm{\b_i} \ge \frac{d}{2}+5$. As a result, we have 
\begin{equation}\label{lowgamma3}
\begin{aligned}
|\a| + 2|\b| 
&\ge |\a_{i}| + |\a_{3}| + 2(|\b_{i}| +|\b_{3}|)  
\\
&\ge \frac{d}{2}+5 + |\s_{3}| -\lfloor\frac{d}{2}+1 \rfloor   +2|\g_{3}| 
\\&\ge |\s_{3}| + 2|\g_{3}| +4.
\end{aligned}
\end{equation}
Therefore, in this case, we bound 
$$
\normm{f}_{\norm{\sigma_{3}},\norm{\gamma_{3}},-N+\norm{\a}+2\norm{\b},D} \le \normm{f}_{\norm{\sigma_{3}},\norm{\gamma_{3}},-N+\norm{\s_3}+2\norm{\g_3},D} \lesssim d_{N}[f]
$$
upon recalling that $|\s_3|+|\g_3|\le N$, see Remark \ref{rem-abi}. This completes the proof of the claim \eqref{claima123}. 

In order to prove \eqref{claimbd}, we need to improve the bounds on the right-hand side to include the energy norm $e_N[f]$. In the case when $|\alpha_i|+|\beta_i|<\frac{d}{2}+5$, using \eqref{equivD}, we have
\begin{align}\label{dbdeN}
\normm{f}_{\norm{\s_{i}},\norm{\g_{i}},r,D} \lesssim \normm{f}_{\norm{\s_{i}},\norm{\g_{i}},r+\frac{1}{2}, L^2} + \normm{f}_{\norm{\s_{i}},\norm{\g_{i}}+1,r+\frac{1}{2}, L^2}  \lesssim e_{N}[f],
\end{align}
since $|\s_i|+ |\g_i| \le d+6 \le N-1$ by Remark \ref{rem-abi}, upon taking $r$ large enough. This, together with \eqref{claima123}, proves \eqref{claimbd} in the case when $|\alpha_i|+|\beta_i|<\frac{d}{2}+5$ for all $i\in \{(1,1),\cdots (1,n),2\}$ (i.e. not including $i=3$). 

It remains to consider the case when $|\alpha_{i_0}|+|\beta_{i_0}|\ge\frac{d}{2}+5$ for some $i_0\in \{(1,1),\cdots (1,n),2\}$. Note that necessarily, due to the partitions \eqref{partitionab}, $|\alpha_i|+|\beta_i| < \frac{d}{2}+5$ for all $i\not =i_0$ (including $i=3$). As a result, \eqref{dbdeN} holds for all $i \not =i_0$. Therefore, in view of the right-hand side of \eqref{ineq:section 5 step1}, it suffices to prove that 
\begin{equation}\label{clgamma3}
\normm{f}_{\norm{\s_{3}},\norm{\g_{3}},-N+\norm{\a}+2\norm{\b},D} \le e_{N}[f]. 
\end{equation}
This indeed holds, since in this case, we note that \eqref{lowgamma3} is valid and therefore, similar to \eqref{dbdeN}, we bound 
\begin{align*}
\normm{f}_{\norm{\s_{3}},\norm{\g_{3}},-N+\norm{\a}+2\norm{\b},D}   
&\lesssim \normm{f}_{\norm{\s_{3}},\norm{\g_{3}},-N+\norm{\s_3}+2\norm{\g_3}+4, L^2} + \normm{f}_{\norm{\s_{3}},\norm{\g_{3}}+1,-N+\norm{\s_3}+2\norm{\g_3}+4, L^2}  
\\&\lesssim e_{N}[f],
\end{align*}
upon noting that $|\sigma_3|+|\gamma_3|\le N-1$ (since $|\alpha_3|+|\beta_3| < \frac{d}{2}+5$ in this case; see Remark \ref{rem-abi}). This completes the proof of \eqref{claimbd}. 

\subsection*{STEP 2) Pure spacial derivative control.}

We assume that the initial perturbations satisfy 
\begin{equation} 
\begin{aligned}
\iint_{\T^d\times\R^d} \begin{pmatrix}
    1 \\
    v_{i} \\ 
    \norm{v}^2
\end{pmatrix}
f(0,x,v) \dd v \dd x =& 0 \;\;  \text{ for all }  i\in\{1,\cdots d\}.
\end{aligned}
\end{equation}
Multiplying
$\begin{pmatrix}
    1 \\
    v_{i} \\ 
    \norm{v}^2
\end{pmatrix}$ to \eqref{LBeqs}, 
we obtain 
\begin{equation*}
\begin{pmatrix}
    1 \\
    v_{i} \\ 
    \norm{v}^2
 \end{pmatrix} \pt_t f + \begin{pmatrix}
    1 \\
    v_{i} \\ 
    \norm{v}^2
\end{pmatrix} v\cdot \nabla f = \begin{pmatrix}
    1 \\
    v_{i} \\ 
    \norm{v}^2
\end{pmatrix}(L[f]+N(f))
\end{equation*}
Since $L[f]$ and $N(f)$ are perpendicular to the kernel of $L$, the mass, momentum and energy conservations remain to valid for all times, namely \eqref{0-f-a-conservation-law}. Thus, by Proposition \ref{proposition:0-macro-L2-estimate}, we have
\begin{equation}\label{Pf - l2 estimate final}
\begin{aligned}
\int_{s}^{t} \normm{ \P \pt^{\a} f}_{L^{2}_{x,v}}^2\dd \tau &\lesssim \; 
\big[ G_{0} (t)-  G_{0} (0)\big]
+  \int_{0}^{t} \normm{\ip \pt^{\a} f}_D^{2} \; ds
 \\&\quad  +\normmm{f}(t)^{4}(1+\normmm{f}(t))^{2N},
\end{aligned}
\end{equation}
where $\normmm{f}(t)$ is defined as in \eqref{defBignorm}. Note that $\normmm{f}(t)$ is an increasing function in $t$. According to the unweighted energy estimate in Lemma \ref{weighted energy estimate} we have
\begin{equation} \label{energy estimate final 1}
\normm{\pt^\a f}^2_{L^2_{x,v}}(t)  + \int_{0}^{t} \normm{\ip \pt^{\a} f}_{D}^2
 \dd s \lesssim \normmm{f_{0}}^2 +\normmm{f}(t)^{4}(1+\normmm{f}(t))^{2N}.
\end{equation}
Combining \eqref{Pf - l2 estimate final} and  \eqref{energy estimate final 1}, we derive
\begin{equation} \label{energy estimate final pure spacial 1}
\normm{\pt^\a f}^2_{L^2_{x,v}}(t)  + \int_{0}^{t} \normm{\pt^{\a} f}_{D}^2
 \dd s \lesssim \normmm{f_{0}}^2 +\normmm{f}(t)^{4}(1+\normmm{f}(t))^{2N}.
\end{equation}
Next, applying the weighted energy estimate for $\pt^{\a}f$ in Lemma \ref{weighted energy estimate}, we thus obtain 
\begin{equation} \label{energy estimate final pure spacial 2}
\normm{\inn{v}^{N- \norm{\a}}\pt^\a f}^2_{L^2_{x,v}}(t)  + \int_{0}^{t} \normm{\inn{v}^{N- \norm{\a}} \pt^{\a} f}_{D}^2
 \dd s \lesssim \normmm{f_{0}}^2 +\normmm{f}(t)^{4}(1+\normmm{f}(t))^{2N}.
\end{equation}

\subsection*{STEP 3) Close the bootstrap.}

To close the nonlinear iterative scheme, we need to derive weighted estimates for derivatives. Indeed, we recall the results from 
Lemma \ref{weighted energy estimate}
that 
\begin{equation*}
\begin{aligned}
\pt_{t} \normm{\inn{v}^{-\th}\pt^\a_{\b} f}^2_{L^2_{x,v}} + \normm{\inn{v}^{-\th} \pt^\a_{\b} f}_{D}^2 
\lesssim&  \normm{\inn{v}^{-\th+1} \pt^{\a+1}_{\b-1} f}_{D}^2+ \sum_{\norm{\b'} < \norm{\b}}\normm{\inn{v}^{-\th} \pt^{\a}_{\b'} f}_{D}^2
\\&+ \int_{\T^d \times \R^d} \inn{v}^{-2 \th}\pt^\a_{\b} f \pt^\a_{\b} g \dd x \dd v.
\end{aligned}
\end{equation*}
Take $\th=-N + \norm{\a} + 2\norm{\b}$. Since $\th -1 = -N + \norm{\a} +1 + 2(\norm{\b}-1)$, the $v$-weight in $\normm{\inn{v}^{-\th+1} \pt^{\a+1}_{\b-1} f}_{D}$ is consistent with those in $\normmm{\cdot}$. Proceeding by induction in $|\beta|$ and using the pure spatial derivative estimates in \eqref{energy estimate final pure spacial 2}, we obtain 
$$
\begin{aligned}
&\normm{\inn{v}^{N- \norm{\a}- 2\norm{\b}}\pt^\a_{\b} f}^2_{L^2_{x,v}}(t)  + \int_{0}^{t} \normm{\inn{v}^{N- \norm{\a}- 2\norm{\b}} \pt^{\a}_{\b} f}_{D}^2
 \dd s 
\lesssim \normmm{f_{0}}^2 +\normmm{f}(t)^{4}(1+\normmm{f}(t))^{2N},
\end{aligned}
$$
for any $|\alpha|+|\beta|\le N$. Combining and recalling \eqref{defBignorm}, we find universal constants $C_0,C_1$ so that 
\begin{equation}
\normmm{f}^2(t) \le C_0 \normmm{f_{0}}^2 +C_1\normmm{f}(t)^{4}(1+\normmm{f}(t))^{2N}
\end{equation}
for any $t\ge 0$, provided that the right hand side remains finite. By the standard continuous induction, this yields $\normmm{f}^2(t) \le C_0 \normmm{f_{0}}^2$ for all $t\ge 0$, provided that the initial perturbations $\normmm{f_{0}}$ are sufficiently small, leading to the existence of a global-in-time solution to \eqref{LBeqs}. This ends the proof of Theorem \ref{theo-main}. 

\section{Proof of the Theorem 2}\label{pf of thm2}

In this section, we prove our main results stated in Theorem \ref{theo-2}. The proof of convergence to equilibrium now follows similarly to that for the homogenous case provided in Chapter 3 of \cite{duerinckx2023well}. We shall thus only detail the modification to the present inhomogeneous setting. Indeed, as done in \cite{StrainGuo}, we consider the mixed weight function 
\begin{align*}
w_{l,\th,K}(v) := \inn{v}^{l} \exp (K \inn{v}^{\th}).
\end{align*}
Let parameters $l,\th,K \ge 0$ be fixed either with $\th <2$, or with $\th =2$ and $K \ll_{V} 1$. We split the proof into two steps.

\subsection*{STEP 1) Compactness estimates.}
Assume that initial data $f_0$ satisfies
\begin{align*}
\iint_{\T^d\times \R^d} w^2_{l,\th,K}|f_0|^2 \; \dd x \dd v< \infty,
\end{align*}
and assume in addition that $\normmm{f_0} \ll_{V,l,\th,K} 1$. In this step we will prove the solution $f$ of \eqref{LBeqs} satisfies
\begin{align} \label{ineq:step1}
\iint_{\T^d\times \R^d} w^2_{l,\th,K}|f(t)|^2\; \dd x \dd v \lesssim \iint_{\T^d\times \R^d} w^2_{l,\th,K}|f_0|^2\; \dd x \dd v.
\end{align}
Similar to \cite{duerinckx2023well}, we focus only on the $\th < 2$ case. For the $\th = 2$ case, one can take the limit $\th \uparrow 2$ utilizing $K \ll 1$ and the uniform bounds. Multiplying \eqref{LBeqs} by $w_{l,\th,K}^2 f$ and integrating over $v$ and $x$, we obtain the following:
\begin{align} \label{eq1:thm2pf}
\frac{1}{2} \pt_t \iint_{\T^d \times \R^d} w_{l,\th,K}^2 |f|^2 \; \dd x \dd v+ I_1(f) =  I_2(f) + I_3(f) +I_4(f)
\end{align}
where
\begin{align*}
I_1(g) :=& \iint_{\T^d \times \R^d} (\nabla_{v}+v) (w_{l,\th,K}^2 g) \cdot A  (\nabla_{v}+v)g\; \dd x \dd v, \\
I_2(g) :=& \iint_{\T^d \times \R^d} \sqrt{\mu}(\nabla_{v}+v) (w_{l,\th,K}^2 g) \cdot \mathcal{B}(\nabla F_g)[(\nabla_v +v)g]\; \dd x \dd v, \\
I_3(g) :=& -\iint_{\T^d \times \R^d} (\nabla_{v}+v) (w_{l,\th,K}^2 g) \cdot \left(\mathcal{B}(\nabla F_g)[g]\nabla g - g\mathcal{B}(\nabla F_g)[\nabla g] \right)\; \dd x \dd v, \\
I_4(g) :=& \iint_{\T^d \times \R^d} (\nabla_{v}+v) (w_{l,\th,K}^2 g) \cdot \left(\mathcal{B}(\nabla F_g)[\sqrt{\mu}] - \mathcal{B}(\nabla \mu)[\sqrt{\mu}](\nabla_v +v)g \right) \; \dd x \dd v.
\end{align*}
Define $\normmb{g}_{l,\th,K}$ as
\begin{align*}
\normmb{g}_{l,\th,K} := \iint_{\T^d \times \R^d} w_{l,\th,K}^2 (\nabla g \cdot A \nabla g+ v g \cdot A vg)\; \dd x \dd v.
\end{align*}
Similarly as done in the sub-Step 1.1 and sub-Step 1.2 in \cite{duerinckx2023well}, we have
\begin{align*}
\normm{\inn{v}^{-\frac{1}{2}} w_{l,\th,K} g}_{L^2} +\normm{\inn{v}^{-\frac{3}{2}} w_{l,\th,K} \nabla g}_{L^2} \approx \normmb{g}_{l,\th,K},
\end{align*}
and
\begin{align} \label{eq2:thm2pf}
I_1(g) \ge \frac{1}{2}\normmb{g}_{l,\th,K}^2 - C_{V,l,\th,K} \normm{g}_{D}^2.
\end{align}
On the other hand, similar to the sub-Step 1.3 in \cite{duerinckx2023well}, we have
\begin{align} \label{eq3:thm2pf}
\norm{I_2(g)} + \norm{I_3(g)} + \norm{I_4(g)} \lesssim \normm{g}_{D}^2 + \normm{g}_{L^\infty_{x}H^2_{v}}\normmb{g}_{l,\th,K}^2.
\end{align}
Using the Sobolev embedding  and noting $d+7 - (\frac{d}{2}+1)-2\times 2 >0$, we bound 
\begin{align} \label{eq4:thm2pf}
\normm{g}_{L^\infty_{x}H^2_{v}} \lesssim\normm{g}_{H^{\frac{d}{2}+1}_{x}H^{2}_{v}} \lesssim \mathscr{E}_{d+7}^{\frac{1}{2}}[g] \ll 1
\end{align}
Combining \eqref{eq1:thm2pf}, \eqref{eq2:thm2pf}, \eqref{eq3:thm2pf} and \eqref{eq4:thm2pf} we have
\begin{align*}
\frac{1}{2} \pt_t \iint_{\T^d\times \R^d} w^2_{l,\th,K}|f(t)|^2 \; \dd x \dd v+ \frac{1}{4} \normmb{f}_{l,\th,K}^2 \lesssim \normm{f}_{D}^2 .
\end{align*}
Recalling $\normmm{f} \ll 1$ and combining Proposition \ref{proposition:0-macro-L2-estimate}, Lemma \ref{weighted energy estimate}, Lemma \ref{lemma:T - upper bound} and Lemma \ref{lemma:S - upper bound}, we can deduce
\begin{align*}
\frac{1}{2} \pt_t \normm{f}_{L^2_{x,v}}^2 + \frac{1}{C} \normm{f}_{D}^2 \le 0,
\end{align*}
which implies \eqref{ineq:step1}

\subsection*{STEP 2) Proof of time decay.}
Given $\e>0$, we have
\begin{align*}
\normm{g}_{D} \gtrsim \iint_{\T^d\times \R^d} \inn{v}^{-1}\norm{g}^2 \; \dd x \dd v\ge \inn{t}^{-\e} \iint_{\T^d\times \R^d} \mathbf{1}_{\inn{v} \le \inn{t}^\e}\norm{g}^2\; \dd x \dd v.
\end{align*}
Thus we have
\begin{align*}
\pt_t \iint_{\T^d \times \R^d} |f|^2 \; \dd x \dd v+ \frac{1}{C} \inn{t}^{-\e} \iint_{\T^d \times \R^d} |f|^2 \; \dd x \dd v\le  \frac{1}{C} \inn{t}^{-\e} \iint_{\T^d \times \R^d} \mathbf{1}_{\inn{v} > \inn{t}^\e}|f|^2\; \dd x \dd v.
\end{align*}
By integration we have
\begin{align*}
\iint_{\T^d \times \R^d} |f(t)|^2 \; \dd x \dd v\le& e^{-\frac{1}{1-\e} \frac{1}{C}\inn{t}^{1-\e}} \iint_{\T^d \times \R^d} |f_0|^2\; \dd x \dd v\\
&+\frac{1}{C} \int_{0}^{t} \inn{s}^{-\e} e^{-\frac{1}{1-\e} \frac{1}{C}(\inn{t}^{1-\e} - \inn{s}^{1-\e})} \iint_{\T^d \times \R^d} \mathbf{1}_{\inn{v} > \inn{t}^\e}|f|^2 \; \dd x \dd v\dd s.
\end{align*}
Appealing to \eqref{ineq:step1} in form of
\begin{align*}
\iint_{\T^d\times \R^d}  \mathbf{1}_{\inn{v} > \inn{t}^\e}|f|^2 \; \dd x \dd v\lesssim \inn{s}^{-2\e l}e^{-2K \inn{s}^{\e \th}} \iint_{\T^d\times \R^d}  w_{l,\th,K}^2|f_0|^2\; \dd x \dd v,
\end{align*}
we deduce
\begin{align*}
\iint_{\T^d \times \R^d} |f(t)|^2 \; \dd x \dd v\le \left(e^{-\frac{1}{1-\e}\frac{1}{C}\inn{t}^{1-\e}} + C \int_{0}^{t}  e^{-\frac{1}{1-\e} \frac{1}{C}(\inn{t}^{1-\e} - \inn{s}^{1-\e})} \inn{s}^{-2\e l -\e}e^{-2K \inn{s}^{\e \th}} \dd s \right)\\
\times \iint_{\T^d \times \R^d} w_{l,\th,K}^2|f_0|^2 \; \dd x \dd v.
\end{align*}
This yields Theorem \ref{theo-2}.

%\bibliography{reference}{}
%\bibliographystyle{plain}

\end{document}